\documentclass[12pt, showkeys]{amsart}%
\usepackage{color}
\usepackage{graphicx}
\usepackage{subfigure}
\usepackage{amssymb}
\usepackage{amsmath,amsxtra}
\usepackage{multirow}
\usepackage{enumerate}
\usepackage{epstopdf}
\usepackage{cite,stmaryrd,txfonts}
\usepackage{tikz}
\usepackage{pgf}
\usetikzlibrary{snakes}
\usetikzlibrary{angles,calc,patterns}
\usetikzlibrary{shapes,arrows,automata}
\usepackage{bbm}
\usepackage{dsfont}
\usepackage{accents}
 \usepackage{float}

\usepackage{epic}
\usepackage{empheq}
\usepackage{cases}
\usepackage{diagbox}
\def\dx{\,{\rm dx}}

\newtheorem{theorem}{Theorem}[section]
\newtheorem{remark}[theorem]{Remark}

\newtheorem{lemma}[theorem]{Lemma}

\newtheorem{definition}[theorem]{Definition}

\newcounter{mnote}
\let\oldmarginpar\marginpar
\renewcommand\marginpar[1]{\-\oldmarginpar[\raggedleft\footnotesize #1]
  {\raggedright\footnotesize #1}}

\usepackage{footnote}
\usepackage{booktabs}

\usepackage{rotating}

\numberwithin{equation}{section}

\usepackage{cite}

\usepackage[colorlinks,linkcolor=black,
anchorcolor=black,
citecolor=black]{hyperref}
\usepackage[shortlabels]{enumitem}
\setlist[enumerate]{nosep}
\usepackage{color, soul}
\soulregister\cite7
\usepackage{setspace}

\def\utau{\undertilde{\tau}}

\def\rot{{\rm rot}}
\def\curl{{\rm curl}}

\def\dv{{\rm div}}

\def\omp{\ominus^\perp}
\def\opp{\oplus^\perp}

\def\od{\mathbf{d}}
\def\oT{\mathbf{T}}

\def\odelta{\boldsymbol{\delta}}
\def\okappa{\boldsymbol{\kappa}}

\def\sX{\xX}
\def\sY{\yY}

\def\icr{{\sf icr}}

\def\hf{\boldsymbol{\mathfrak{H}}}

\def\R{\mathcal{R}}
\def\N{\mathcal{N}}

\usepackage[mathscr]{eucal}

\def\xD{\boldsymbol{\mathbf{D}}}

\def\xM{\boldsymbol{\mathbf{M}}}

\def\xX{{\boldsymbol{\mathbf{X}}}}
\def\xY{{\boldsymbol{\mathbf{Y}}}}

\def\xv{\boldsymbol{\mathbf{v}}}
\def\xw{\boldsymbol{\mathbf{w}}}

\def\yD{\boldsymbol{\mathbbm{D}}}

\def\yN{\boldsymbol{\mathbbm{N}}}

\def\yT{\boldsymbol{\mathbbm{T}}}

\def\yY{{\boldsymbol{\mathbbm{Y}}}}

\def\yv{\boldsymbol{\mathbbm{v}}}
\def\yw{\boldsymbol{\mathbbm{w}}}

\def\asD{\yD}

\def\tyN{\yN}

\def\sD{\xD}

\def\avv{\yv}
\def\avw{\yw}

\def\vv{\xv}
\def\vw{\xw}

\def\aoT{\yT}

\def\txM{\widetilde{\xM}}
\def\uxM{{\undertilde{\xM}}}

\def\tyN{\mathbf{\widetilde{{\mathbbm{N}}}}}
\def\uyN{{\undertilde{\yN}}}

\def\fomega{\boldsymbol{\omega}}
\def\fmu{\boldsymbol{\mu}}

\def\fzeta{\boldsymbol{\zeta}}
\def\feta{\boldsymbol{\eta}}

\def\fsigma{\boldsymbol{\sigma}}
\def\ftau{\boldsymbol{\tau}}

\def\fvartheta{\boldsymbol{\vartheta}}
\def\fvarsigma{\boldsymbol{\varsigma}}

\def\fpsi{\boldsymbol{\psi}}

\def\fW{\boldsymbol{W}}
\def\fV{\boldsymbol{V}}

\def\ff{\boldsymbol{f}}

\def\ixalpha{\boldsymbol{\alpha}} 
\def\ixbeta{\boldsymbol{\beta}}

\def\pddemL{\mathbf{P}_{\od\cap\odelta}^{\rm m}\Lambda}
\def\pddempdL{\mathbf{P}_{\od\cap\odelta}^{\rm m+\od}\Lambda}
\def\pddempdeL{\mathbf{P}_{\od\cap\odelta}^{\rm m+\odelta}\Lambda}

\def\pddedmpdeL{\mathbf{P}_{\od\cap\odelta,\od}^{\rm m+\odelta}\Lambda}
\def\pddedempdeL{\mathbf{P}_{\od\cap\odelta,\odelta}^{\rm m+\odelta}\Lambda}

\def\pddemL{\mathbf{P}_{\od\cap\odelta}^{\rm m}\Lambda}

\begin{document}

\title{A primal finite element scheme of the Hodge Laplace problem}

\author{Shuo Zhang}
\address{LSEC, Institute of Computational Mathematics and Scientific/Engineering Computing, Academy of Mathematics and System Sciences, Chinese Academy of Sciences, Beijing 100190; University of Chinese Academy of Sciences, Beijing, 100049; People's Republic of China}
\email{szhang@lsec.cc.ac.cn}

\thanks{The research is partially supported by NSFC (11871465) and CAS (XDB41000000).}


%
%
%
%
%
%
%


\begin{abstract}
In this paper, a unified family, for any $n\geqslant 2$ and $1\leqslant k\leqslant n-1$, of nonconforming finite element schemes are presented for the primal weak formulation of the $n$-dimensional Hodge-Laplace equation on $H\Lambda^k\cap H^*_0\Lambda^k$ and on the simplicial subdivisions of the domain. The finite element scheme possesses an $\mathcal{O}(h)$-order convergence rate for sufficiently regular data, and an $\mathcal{O}(h^s)$-order rate on any $s$-regular domain, $0<s\leqslant 1$, no matter what topology the domain has. 
\end{abstract}

\maketitle

\tableofcontents

\section{Introduction}

Let $\Omega\subset\mathbb{R}^n$ be a domain with Lipschitz boundary. In this paper, we consider the primal weak formulation of the Hodge-Laplace problem: given $\ff\in L^2\Lambda^k(\Omega)$, find $\fomega\in H\Lambda^k(\Omega)\cap H^*_0\Lambda^k(\Omega)$, such that 
\begin{equation}\label{eq:modelhlori}
\left\{
\begin{array}{c}
\langle\fomega,\fvartheta\rangle_{L^2\Lambda^k}=0,\ \ \forall\,\fvartheta\in \hf\Lambda^k,\ \mbox{and}
\\  
\langle\od^k\fomega,\od^k\fmu\rangle_{L^2\Lambda^{k+1}}+\langle\odelta_k\fomega,\odelta_k\fmu\rangle_{L^2\Lambda^{k-1}}=\langle\ff-\mathbf{P}_{\boldsymbol{\mathfrak{H}}}\ff,\fmu\rangle_{L^2\Lambda^k},\ \ \forall\,\fmu\in H\Lambda^k(\Omega)\cap H^*_0\Lambda^k(\Omega).
\end{array}\right.
\end{equation}
Here, following \cite{Arnold.D2018feec}, we denote by $\Lambda^k(\Xi)$ the space of differential $k$-forms on an $n$-dimensional domain $\Xi$, and $L^2\Lambda^k(\Xi)$ consists of differential $k$-forms with coefficients in $L^2(\Xi)$ component by component, and $\langle\cdot,\cdot\rangle_{L^2\Lambda^k(\Xi)}$ is the inner product of the Hilbert space $L^2\Lambda^k(\Xi)$. In the sequel, we will occasionally drop $\Omega$ for differential forms on $\Omega$. The exterior differential operator $\od^k:\Lambda^k(\Xi)\to \Lambda^{k+1}(\Xi)$ is an unbounded operator from $L^2\Lambda^k(\Xi)$ to $\Lambda^{k+1}(\Xi)$. Denote
$$
H\Lambda^k(\Xi):=\left\{\fomega\in L^2\Lambda^k(\Xi):\od^k\fomega\in L^2\Lambda^{k+1}(\Xi)\right\},
$$
and $H\Lambda^k(\Xi)$ is a Hilbert space with the norm $\|\fomega\|_{L^2\Lambda^k(\Xi)}+\|\od^k\fomega\|_{L^2\Lambda^{k+1}(\Xi)}$. Denote by $H_0\Lambda^k(\Xi)$ the closure of $\mathcal{C}_0^\infty\Lambda^k(\Xi)$ in $H\Lambda^k(\Xi)$. The Hodge star operator $\star$ maps $L^2\Lambda^k(\Xi)$ isomorphically to $L^2\Lambda^{n-k}(\Xi)$ for each $0\leqslant k\leqslant n$. The \emph{codifferential operator} $\odelta_k$ defined by $\odelta_k\fmu=(-1)^{kn}\star\od^{n-k}\star\fmu$ is unbounded from $L^2\Lambda^k(\Xi)$ to $L^2\Lambda^{k-1}(\Xi)$.  Denote 
$$
H^*\Lambda^k(\Xi):=\left\{\fmu\in L^2\Lambda^k(\Xi):\odelta_k\fmu\in L^2\Lambda^{k-1}(\Xi)\right\},
$$
and $H^*_0\Lambda^k(\Xi)$ the closure of $\mathcal{C}_0^\infty\Lambda^k(\Xi)$ in $H^*\Lambda^k(\Xi)$. Then $H^*\Lambda^k(\Xi)=\star H\Lambda^{n-k}(\Xi)$, and $H^*_0\Lambda^k(\Xi)=\star H_0\Lambda^{n-k}(\Xi)$. Denote spaces of harmonic forms by 
$$
\hf_{(0)}\Lambda^k(\Xi):=\N(\od^k,H_{(0)}\Lambda^k(\Xi))\omp\R(\od^{k-1},H_{(0)}\Lambda^{k-1}(\Xi)),
$$
where $\N(\cdot,\cdot)$ and $\R(\cdot,\cdot)$ denote the kernel and range spaces of certain operators, and $\omp$ denotes the orthogonal difference, namely  $\N(\od^k,H_{(0)}\Lambda^k(\Xi))=\R(\od^{k-1},H_{(0)}\Lambda^{k-1}(\Xi))\opp \hf_{(0)}\Lambda^k(\Xi)$. Similarly 
$$
\hf_{(0)}^*\Lambda^k(\Xi):=\N(\odelta_k,H^*_{(0)}\Lambda^k(\Xi))\omp\R(\odelta_{k+1},H^*_{(0)}\Lambda^{k+1}(\Xi)).
$$ 
Then $\star\hf_0\Lambda^k=\hf^*_0\Lambda^{n-k}$, and further the Poincar\'e-Leftschetz duality holds as $\hf\Lambda^{n-k}=\star\hf_0\Lambda^k$. Besides, $\mathbf{P}_{\hf}$ denotes the $L^2$ projection to $\hf\Lambda^k$.

The model problem \eqref{eq:modelhlori} corresponds to a strong form that 
\begin{equation}
\od^k\fomega\in H^*_0\Lambda^{k+1}(\Omega),\ \ \ \odelta_k\fomega\in H\Lambda^{k-1}(\Omega),
\end{equation}
and
\begin{equation}
\fomega\perp\hf\Lambda^k(\Omega),\quad\mbox{and}\quad
\odelta_{k+1}\od^k\fomega+\od^{k-1}\odelta_k\fomega=\ff-\mathbf{P}_{\boldsymbol{\mathfrak{H}}}\ff. 
\end{equation}

The Hodge-Laplace problem arises in many applied sciences, including electromagnetics\cite{Monk.P2003mono,Hiptmair.R2002acta}, fluid-structure interaction \cite{Bathe.K;Nitikitpaiboon.C;Wang.X1995cs,Bermudez.A;Rodriguez.R1994cmame,Hamdi.M;Ousset.Y;Verchery.G1978ijnme}, and others. Particularly, the numerical solution of the Hodge Laplace equation is a central subject of the theory of finite element exterior calculus (FEEC), and we refer to \cite{Arnold.D;Falk.R;Winther.R2006acta,Arnold.D;Falk.R;Winther.R2010bams,Arnold.D2018feec} for a thorough introduction to FEEC. 
~\\

A major feature in the discretization of Hodge Laplace problem is that, the conforming finite element scheme for \eqref{eq:modelhlori} may lead to a spurious solution that converges to a wrong limit when the exact solution $\fomega$ is not regular enough. Indeed, as is well known, for domains which are not smooth enough, the singular part of $\fomega$ can not be captured by the conforming finite element space. To cope with this situation, a well-developed approach is to use mixed finite element method. Again, the main approach can be found in detail in \cite{Arnold.D;Falk.R;Winther.R2006acta,Arnold.D;Falk.R;Winther.R2010bams,Arnold.D2018feec}, for which the structure of de Rham complex plays a crucial role, and another key ingredient is that spaces of discrete harmonic forms are established isomorphic to the space of continuous harmonic forms. Besides, some recent progress can be found in \cite{Li.Y2019sinum,Demlow.A;Hirani.A2014} for {\it a posteriori} error estimation and adaptive methods, and in \cite{Hong.Q;Li.Y;Xu.J2022mc} for a detailed analysis of Discontinuous Galerkin (DG) methods in FEEC in the newly-presented eXtended Galerkin (XG) framework.

On the other hand, to discretize directly the primal formulation \eqref{eq:modelhlori} has been still attracting research interests. Virtual element methods are designed for the three dimensional vector potential formulation of magnetostatic problems \cite{Beiroa.L;Brezzi.F;Marini.D;Alessandro.R2018}, with the major interests restricted to cases  where the computational domain has no re-entrant corner, and the space of harmonic forms is not concerned for these cases. Nonconforming element methods and discontinuous Galerkin methods are also designed which can lead to a correct approximation of the nonsmooth solution for the $H(\curl)\cap H(\dv)$ problem in two dimension polygonal domains, particularly for \eqref{eq:modelhlori} on domains with connected boundary on which harmonic forms vanish; readers are referred to \cite{Brenner.S;Sung.L;Cui.J2008} for an interior penalty method, to \cite{brenner2009quadratic} for a nonconforming finite element method, and  to \cite{Brenner.S;Cui.J;Li.F;Sung.L2008nm} for a nonconforming finite element used with inter-element penalties. Recent works also include \cite{Barker.M2022thesis,Barker.M;Cao.S;Stern.A2022arxiv,Mirebeau.J2012aml}.
~\\

In this paper, we present a unified family, for any $n\geqslant 2$ and $1\leqslant k\leqslant n-1$, of nonconforming finite element schemes for the primal formulation \eqref{eq:modelhlori} and on the subdivision of the domain by simplexes. The main feature of the finite element schemes is a nonconforming finite element space for $H\Lambda^k\cap H^*_0\Lambda^k$ where all the finite element functions are defined by local shape function spaces and the continuity conditions, while for $\hf\Lambda^k$, we use the well-studied discrete space of harmonic forms as, e.g., in \cite{Arnold.D;Falk.R;Winther.R2006acta,Arnold.D;Falk.R;Winther.R2010bams,Arnold.D2018feec}, and no penalty term or stabilization is used in the schemes. The local shape function space is a slight enrichment based on \ $\mathcal{P}_0\Lambda^k(T)+\okappa(\mathcal{P}_0\Lambda^{k+1}(T))+\star\okappa\star(\mathcal{P}_0\Lambda^{k-1}(T))$, namely the minimal local space for $\odelta_{k+1}\od^k+\od^{k-1}\odelta_k$, but, not similar to \cite{Brenner.S;Sung.L;Cui.J2008,Brenner.S;Cui.J;Li.F;Sung.L2008nm,brenner2009quadratic,Barker.M2022thesis,Barker.M;Cao.S;Stern.A2022arxiv,Mirebeau.J2012aml}, it does not contain the complete linear polynomial space. Another difference from these existing works, particularly \cite{brenner2009quadratic,Mirebeau.J2012aml}, is that the finite element functions in this present paper possess a different kind of inter-element continuity. As precisely described in \eqref{eq:deffems}, the continuity is imposed in a dual way; the dual way for imposing continuity has been suggested by the theory of partially adjoint operators in \cite{Zhang.S2022padao-arxiv}, and used for constructing nonconforming $H(\od)$ Whitney form spaces as well as commutative diagrams. The way makes the finite element functions not correspond to a ``finite element" in the sense of Ciarlet's triple \cite{Ciarlet.P1978book}, and the analysis thus relies on non-standard techniques. In this paper, for the analysis, the discrete Poincar\'e inequality, which is crucial with respect to nontrivial topologies, is proved by the theory of partially adjoint operators developed in \cite{Zhang.S2022padao-arxiv}, and different from \cite{Zhang.S2022padao-arxiv}, the error estimation is accomplished by an indirect way; namely, we first show as Lemma \ref{lem:disequivalent} that certain primal scheme \eqref{eq:modelhldisff0} is, in some sense, equivalent to a classical mixed element scheme \eqref{eq:classdisf0}, and the error estimation of the classical mixed scheme can be used as a midway step for the analysis. We finally show that, the finite element scheme possesses an $\mathcal{O}(h)$-order convergence rate for sufficiently regular data, and an $\mathcal{O}(h^s)$-order rate on any $s$-regular domain, $0<s\leqslant 1$, no matter the topology is trivial or not. 
~\\

It is interesting to clarify again that, the schemes given in this paper are primal ones, even though the continuity conditions for the finite element functions are imposed in a dual way, and for some special cases, the schemes can be equivalent, in some sense, to mixed schemes. In principle, the finite element scheme aims at \eqref{eq:modelhlori} (as the primal weak formulation (4.15) of \cite{Arnold.D2018feec}) and utilizes only a single field. Meanwhile, the equivalence between primal and mixed finite element schemes has been found as to, e.g., the Poisson and the biharmonic equations\cite{Marini.L1985sinum,Arnold.D;Brezzi.F1985}. For practical implementation, in the present paper, a precise set of basis functions can be figured out for the newly-designed finite element space, the supports of which are each contained in a vertex patch, and the programming of the scheme can be done in a standard routine as for the standard ``primal finite element" method. The figuration of basis functions is quite similar to the procedure given in \cite{Zhang.S2022padao-arxiv}, but also with essentially different steps. The basis functions are presented in a unified way, and an illustration is given for the two-dimensional case for example. Locally supported basis functions have been also found and implemented for many other specific non-Ciarlet type finite element spaces\cite{Fortin.M;Soulie.M1983,Park.C;Sheen.D2003,Zhang.S2020IMA,Zhang.S2021SCM,Liu.W;Zhang.S2022arxiv,Zeng.H;Zhang.C;Zhang.S2020arxiv,Xi.Y;Ji.X;Zhang.S2021cicp,Xi.Y;Ji.X;Zhang.S2020jsc}.
~\\

The remaining of the paper is organized as follows. In Section \ref{sec:pre}, some preliminaries are collected. Particularly, some key points of the theory of partially adjoint operator, which is developed in \cite{Zhang.S2022padao-arxiv} and is fundamental in the present paper, are reviewed, including the definitions of base operator pair (Definition \ref{def:basepair}) and of partially adjoint operators (Definition \ref{def:pao}), and the quantified closed range theorem for partially adjoint operators (Theorem \ref{thm:chpi}). In Section \ref{sec:fes}, the finite element space is constructed, and the discrete Poincar\'e inequality of the space is proved by the theory of partially adjoint operator. In Section \ref{sec:fescheme}, a unified family of finite element schemes are defined, and the error estimation and the  implementation are presented.

\section{Preliminaries}

\label{sec:pre}

%

%
%
\subsection{Theory of partially adjoint operators}
Let $\xX$ and $\yY$ be two Hilbert spaces with respective inner products $\langle\cdot,\cdot\rangle_{\xX}$ and $\langle\cdot,\cdot\rangle_{\yY}$, and respective norms $\|\cdot\|_{\xX}$ and $\|\cdot\|_{\yY}$. Let $(\oT,\txM):\xX\to \yY$ and $(\aoT,\tyN):\yY\to\xX$ be two closed operators, not necessarily densely defined. Denote, for $\vv\in\txM$, $\|\vv\|_\oT:=(\|\vv\|_\xX^2+\|\oT\vv\|_\yY^2)^{1/2}$, and for $\avv\in\tyN$, $\|\avv\|_{\aoT}:=(\|\avv\|_\yY^2+\|\aoT\avv\|_\xX^2)^{1/2}$. Denote
\begin{equation}\label{eq:uxm}
\uxM:=\left\{\vv\in \txM:\langle\vv,\aoT\avv\rangle_\xX-\langle\oT\vv,\avv\rangle_\yY=0,\ \forall\,\avv\in\tyN\right\},
\end{equation}
\begin{equation}\label{eq:uyn}
\uyN:=\left\{\avv\in \tyN:\langle\vv,\aoT\avv\rangle_\xX-\langle\oT\vv,\avv\rangle_\yY=0,\ \forall\,\vv\in\txM\right\},
\end{equation}
\begin{equation}\label{eq:xmb}
\xM_{\rm B}:=\left\{\vv\in \txM:\langle \vv,\vw \rangle_{\xX}=0,\forall\,\vw\in \N(\oT,\uxM);\ \langle\oT \vv,\oT \vw\rangle_{\yY}=0,\ \forall\,\vw\in \uxM\right\},
\end{equation}
and
\begin{equation}\label{eq:ynb}
\yN_{\rm B}:=\left\{\avv\in \tyN:\langle \avv,\avw \rangle_{\yY}=0,\forall\,\avw\in \N(\aoT,\uyN);\ \langle\aoT \avv,\aoT \avw\rangle_{\xX}=0,\ \forall\,\avw\in \uyN\right\}.
\end{equation}
We call $(\xM_{\rm B},\yN_{\rm B})$ the {\bf twisted} part of $(\txM,\tyN)$. 

\begin{definition}[Definition 2.13 of \cite{Zhang.S2022padao-arxiv}]\label{def:basepair}
A pair of closed operators $\left[(\oT,\txM):\xX\to\yY,(\aoT,\tyN):\yY\to\xX\right]$ is called a {\bf base operator pair}, if, with notations \eqref{eq:uxm}, \eqref{eq:uyn}, \eqref{eq:xmb} and \eqref{eq:ynb}, 
\begin{enumerate}
\item $\R(\oT,\txM)$, $\R(\aoT,\tyN)$, $\R(\oT,\uxM)$ and $\R(\aoT,\uyN)$ are all closed; 
\item $\mathcal{N}(\oT,\xM_{\rm B})$ and $\mathcal{R}(\aoT,\yN_{\rm B})$ are isomorphic, and $\mathcal{N}(\aoT,\yN_{\rm B})$ and $\mathcal{R}(\oT,\xM_{\rm B})$ are isomorphic. 
\end{enumerate}
\end{definition}

For $\left[(\oT,\txM),(\aoT,\tyN)\right]$ a base operator pair, for nontrivial $\R(\aoT,\yN_{\rm B})$ and $\N(\oT,\xM_{\rm B})$, denote 
\begin{equation}\label{eq:alpha}
\displaystyle \alpha:=\inf_{0\neq \vv\in\mathcal{N}(\oT,\xM_{\rm B})}\sup_{\vw\in \mathcal{R}(\aoT,\yN_{\rm B})}\frac{\langle\vv,\vw\rangle_\xX}{\|\vv\|_\xX\|\vw\|_\xX}
=
\inf_{0\neq \vw\in\mathcal{R}(\aoT,\yN_{\rm B})}\sup_{\vv\in\mathcal{N}(\oT,\xM_{\rm B})}\frac{\langle\vv,\vw\rangle_\xX}{\|\vv\|_\xX\|\vw\|_\xX},
\end{equation} 
and for nontrivial $\N(\aoT,\yN_{\rm B})$ and $\R(\oT,\xM_{\rm B})$, denote
\begin{equation}\label{eq:beta}
\beta:=\inf_{0\neq \avv\in\mathcal{N}(\aoT,\yN_{\rm B})}\sup_{\avw\in\mathcal{R}(\oT,\xM_{\rm B})}\frac{\langle\avv,\avw\rangle_\yY}{\|\avv\|_\yY\|\avw\|_\yY}
=
\inf_{0\neq \avw\in\mathcal{R}(\oT,\xM_{\rm B})}\sup_{\avv\in\mathcal{N}(\aoT,\yN_{\rm B})}\frac{\langle\avv,\avw\rangle_\yY}{\|\avv\|_\yY\|\avw\|_\yY}.
\end{equation}
Then $\alpha>0$ and $\beta>0$. We further make a convention that,
\begin{equation}\label{eq:alphabeta=1}
\displaystyle \left\{
\begin{array}{ll}
\alpha=1, & \mbox{if}\, \mathcal{N}(\oT,\xM_{\rm B})=\mathcal{R}(\aoT,\yN_{\rm B})=\left\{0\right\};
\\
\beta=1, & \mbox{if}\, \mathcal{N}(\aoT,\yN_{\rm B})=\mathcal{R}(\oT,\xM_{\rm B})=\left\{0\right\}.
\end{array}
\right.
\end{equation}

\begin{definition}[Definition 2.15 of \cite{Zhang.S2022padao-arxiv}]\label{def:pao}
For $\left[(\oT,\txM):\xX\to \yY,(\aoT,\tyN):\yY\to\xX\right]$ a base operator pair, two operators $(\oT,\sD)\subset(\oT,\txM)$ and $(\aoT,\asD)\subset(\aoT,\tyN)$ are called {\bf partially adjoint} based on $\left[(\oT,\txM),(\aoT,\tyN)\right]$, if
\begin{equation}\label{eq:pacondition}
\displaystyle \sD=\left\{\vv\in \txM:\langle \vv,\aoT \avv\rangle_\sX-\langle \oT \vv,\avv\rangle_\sY=0,\ \forall\,\avv\in\asD\right\},
\end{equation}
and
\begin{equation}
\displaystyle \asD=\left\{\avv\in \tyN: \langle \vv,\aoT \avv\rangle_\sX-\langle \oT \vv,\avv\rangle_\sY=0,\ \forall\,\vv\in \sD\right\}.
\end{equation}
\end{definition}

\begin{definition}[Definition 2.8 of \cite{Zhang.S2022padao-arxiv}]
For $(\oT,\xD):\xX\to \yY$ a closed  operator, denote 
$$
\xD^{\boldsymbol \lrcorner}:=\left\{\xv\in \xD:\langle \xv,\xw\rangle_\xX=0,\ \forall\,\xw\in \mathcal{N}(\oT,\xD)\right\}.
$$ Define the {\bf index of closed range} of $(\oT,\xD)$ as
\begin{equation}\label{eq:deficr}
\mathsf{icr}(\oT,\xD):=\left\{\begin{array}{rl}
\displaystyle \sup_{0\neq\xv\in \xD^{\boldsymbol \lrcorner}}\frac{\|\xv\|_\xX}{\|\oT\xv\|_\yY},&\mbox{if}\ \xD^{\boldsymbol \lrcorner}\neq\left\{0\right\};
\\
0,&\mbox{if}\ \xD^{\boldsymbol \lrcorner}=\left\{0\right\}.
\end{array}\right.
\end{equation}
\end{definition}
Note that $\icr(\oT,\sD)$ evaluates in $\left[0,+\infty\right]$, and $\R(\oT,\sD)$ is closed if and only if $\icr(\oT,\sD)<\infty$. Further, $\icr(\oT,\xD^{\boldsymbol \lrcorner})$ plays like the constant for Poincar\'e inequality in the sense that $\|\vv\|_\xX\leqslant \icr(\oT,\xD^{\boldsymbol \lrcorner})\|\oT\vv\|_\xY$ for $\vv\in\xD^{\boldsymbol \lrcorner}$.

\begin{theorem}[quantified closed range theorem, Theorem 2.21 of \cite{Zhang.S2022padao-arxiv}]\label{thm:chpi}
For $\left[(\oT,\sD),(\aoT,\asD)\right]$ partially adjoint based on $\left[(\oT,\txM),(\aoT,\tyN)\right]$, with notations given in \eqref{eq:uxm}, \eqref{eq:uyn}, \eqref{eq:alpha}, \eqref{eq:beta} and \eqref{eq:alphabeta=1}, if $\icr(\aoT,\asD)<\infty$,
\begin{equation}
\icr(\oT,\sD)\leqslant (1+\alpha^{-1})\cdot\icr(\oT,\txM)+\alpha^{-1}\icr(\aoT,\asD)+\icr(\oT,\uxM);
\end{equation}
if $\icr(\oT,\sD)<\infty$,
\begin{equation}\label{eq:icrtstardstar}
\icr(\aoT,\asD) \leqslant (1+\beta^{-1})\cdot\icr(\aoT,\tyN)+\beta^{-1}\icr(\oT,\sD)+\icr(\aoT,\uyN).
\end{equation}
\end{theorem}

\subsection{Polynomial spaces on a simplex}
Denote the set of $k$-indices as
$$
\mathbb{IX}_{k,n}:=\left\{\ixalpha=(\ixalpha_1,\dots,\ixalpha_k)\in\mathbb{N}^k:1\leqslant \ixalpha_1<\ixalpha_2<\dots<\ixalpha_k\leqslant n,\ \mathbb{N}\ \mbox{the\ set\ of\ integers}\right\}. 
$$

Then 
$$
\od^k(\okappa(\dx^{\ixalpha_1}\wedge\dots\wedge\dx^{\ixalpha_{k+1}}))=(k+1)\dx^{\ixalpha_1}\wedge\dots\wedge\dx^{\ixalpha_{k+1}},
$$
and
$$
\odelta_k(\star\okappa\star(\dx^{\ixalpha_1}\wedge\dots\wedge \dx^{\ixalpha_{k-1}}))=(-1)^{kn-n-1} (n-k+1)(\dx^{\ixalpha_1}\wedge\dots\wedge\dx^{\ixalpha_{k-1}}),
$$
where $\okappa$ is the Koszul operator 
$$
\okappa(\dx^{\ixalpha_1}\wedge\dots\wedge\dx^{\ixalpha_k})=\sum_{j=1}^k(-1)^{j+1}x^{\ixalpha_j}\dx^{\ixalpha_1}\wedge\dots\wedge\dx^{\ixalpha_{j-1}}\wedge\dx^{\ixalpha_{j+1}}\wedge\dots\wedge\dx^{\ixalpha_k}.
$$

Given $T$ a simplex, denote on $T$ $\tilde{x}^j=x^j-c_j$ where $c_j$ is a constant such that $\int_T\tilde{x}^j=0$. Denote a simplex dependent Koszul operator 
$$
\okappa_T(\dx^{\ixalpha_1}\wedge\dots\wedge\dx^{\ixalpha_k}):=\sum_{j=1}^k(-1)^{(j+1)}\tilde{x}^{\ixalpha_j}\dx^{\ixalpha_1}\wedge\dots\dx^{\ixalpha_{j-1}}\wedge\dx^{\ixalpha_{j+1}}\wedge\dots\wedge\dx^{\ixalpha_k},\ \mbox{for}\ \ixalpha\in\mathbb{IX}_{k,n}.
$$
Then
$$
\od^{k-1}\okappa_T(\dx^{\ixalpha_1}\wedge\dots\dx^{\ixalpha_k})=k\dx^{\ixalpha_1}\wedge\dots\dx^{\ixalpha_k}. 
$$
In this part and in the sequel, we make the convention that, for $\ixalpha\in\mathbb{IX}_{k,n}$, we use $\ixbeta$ for one in $\mathbb{IX}_{n-k,n}$, such that $\ixalpha$ and $\ixbeta$ partition $\{1,2,\dots,n\}$. For $\ixalpha\in\mathbb{IX}_{k,n}$, following \cite{Zhang.S2022primalddelta-arxiv}, denote 
$$
\tilde\fmu_{\odelta,T}^{\ixalpha}=\sum_{j=1}^k[(\tilde{x}^{\ixalpha_j})^2-c^{\ixalpha_j}]\dx^{\ixalpha_1}\wedge\dots\wedge\dx^{\ixalpha_k}, 
$$
and
$$
\tilde\fmu_{\od,T}^{\ixalpha}=\sum_{j=1}^{n-k}[(\tilde{x}^{\ixbeta_j})^2-c^{\ixbeta_j}]\dx^{\ixalpha_1}\wedge\dx^{\ixalpha_2}\wedge\dots\wedge\dx^{\ixalpha_k}, 
$$
where $c^{\ixalpha_j}$ and $c^{\ixbeta_j}$ are constants such that $\int_T [(\tilde x^{\ixalpha_j})^2-c^{\ixalpha_j}]=0$ and $\int_T[(\tilde x^{\ixbeta_j})^2-c^{\ixbeta_j}]=0$. Then, (\cite{Zhang.S2022primalddelta-arxiv})
$$
\od^k\tilde\fmu_{\odelta,T}^{\ixalpha}=0,\quad \odelta_k\tilde\fmu_{\odelta,T}^{\ixalpha}=2(-1)^n\okappa_T(\dx^{\ixalpha_1}\wedge\dx^{\ixalpha_2}\wedge\dots\wedge \dx^{\ixalpha_k}),
$$

$$
\odelta_k\tilde\fmu_{\od,T}^{\ixalpha}=0,\quad \mbox{and},\quad \od^k\tilde\fmu_{\od,T}^{\ixalpha}=2(-1)^{n(1+k)+1}\star(\okappa_T(\star (\dx^{\ixalpha_1}\wedge\dots\dx^{\ixalpha_k}))).
$$

Denote, following \cite{Zhang.S2022primalddelta-arxiv},
\begin{equation}\label{eq:defhd}
\mathcal{H}^2_{\od}\Lambda^k(T):={\rm span}\left\{\tilde\fmu_{\od,T}^{\ixalpha}:\ixalpha\in\mathbb{IX}_{k,n}\right\},
\end{equation}
and
\begin{equation}\label{eq:defhdelta}
\mathcal{H}^2_{\odelta}\Lambda^k(T):={\rm span}\left\{\tilde\fmu_{\odelta,T}^{\ixalpha}:\ixalpha\in\mathbb{IX}_{k,n}\right\}.
\end{equation}

\begin{lemma}\label{lem:surjecdde}(\cite{Zhang.S2022primalddelta-arxiv})
\begin{enumerate}
\item $\od^k$ is bijective from $\okappa_T(\mathcal{P}_0\Lambda^{k+1})$ onto $\mathcal{P}_0\Lambda^{k+1}$, and bijective from $\mathcal{H}^2_{\od}\Lambda^k(T)$ onto $\star\okappa_T\star(\mathcal{P}_0\Lambda^k(T))$.
\item $\odelta_k$ is bijective from $\star\okappa_T\star(\mathcal{P}_0\Lambda^{k-1})$ onto $\mathcal{P}_0\Lambda^{k-1}$, and bijective from $\mathcal{H}^2_{\odelta}\Lambda^k(T)$ onto $\okappa_T(\mathcal{P}_0\Lambda^k(T))$.
\end{enumerate}
\end{lemma}

\begin{lemma}\label{lem:hfreq}\cite{Zhang.S2022primalddelta-arxiv}
There exists a constant $C_{k,n}$, depending on the regularity of $T$, such that
\begin{equation}
\|\fmu\|_{L^2\Lambda^k(T)}\leqslant C_{k,n}h_T\|\odelta_k\fmu\|_{L^2\Lambda^{k-1}(T)},\ \ \mbox{for}\ \fmu\in \star\okappa_T\star(\mathcal{P}_0\Lambda^{k-1}(T))+\mathcal{H}^2_{\odelta}\Lambda^k(T),
\end{equation}
and
\begin{equation}
\|\fmu\|_{L^2\Lambda^k(T)}\leqslant C_{k,n}h_T\|\od^k\fmu\|_{L^2\Lambda^{k+1}(T)},\ \ \mbox{for}\ \fmu\in \okappa_T(\mathcal{P}_0\Lambda^{k+1}(T))+\mathcal{H}^2_{\od}\Lambda^k(T).
\end{equation}
\end{lemma}

Denote by $\mathbf{P}^k_{0,T}$ the $L^2$ projection onto $\mathcal{P}_0\Lambda^k(T)$.  The lemma follows by Lemma \ref{lem:hfreq} directly. 
\begin{lemma}\label{lem:hfonT}
There exists a constant $C_{k,n}$, depending on the regularity of $T$, such that
\begin{multline}
\|\fmu\|_{L^2\Lambda^k(T)}\leqslant C_{k,n}h_T(\|\odelta_k\fmu\|_{L^2\Lambda^{k-1}(T)}+\|\od^k\fmu\|_{L^2\Lambda^{k-1}(T)}),
\\ 
\mbox{for}\ \fmu\in \star\okappa_T\star(\mathcal{P}_0\Lambda^{k-1}(T))+\okappa_T(\mathcal{P}_0\Lambda^{k+1}(T))+\mathcal{H}^2_{\odelta}\Lambda^k(T)+\mathcal{H}^2_{\od}\Lambda^k(T),
\end{multline}
and
\begin{multline}
\|\fmu-\mathbf{P}^k_{0,T}\fmu\|_{L^2\Lambda^k(T)}\leqslant C_{k,n}h_T(\|\odelta_k\fmu\|_{L^2\Lambda^{k-1}(T)}+\|\od^k\fmu\|_{L^2\Lambda^{k-1}(T)}),
\\ 
\mbox{for}\ \fmu\in \mathcal{P}_0\Lambda^k(T)+\star\okappa_T\star(\mathcal{P}_0\Lambda^{k-1}(T))+\okappa_T(\mathcal{P}_0\Lambda^{k+1}(T))+\mathcal{H}^2_{\odelta}\Lambda^k(T)+\mathcal{H}^2_{\od}\Lambda^k(T).
\end{multline}
\end{lemma}

\subsection{Conforming and nonconforming Whitney forms for exterior differential forms}

The space of Whitney forms, the lowest-degree trimmed polynomial $k$-forms, associated with the operator $\od^k$ is denoted by (\cite{Arnold.D;Falk.R;Winther.R2006acta,Arnold.D;Falk.R;Winther.R2010bams,Arnold.D2018feec})
$$
\mathcal{P}^-_1\Lambda^k=\mathcal{P}_0\Lambda^k+\okappa(\mathcal{P}_0\Lambda^{k+1}).
$$
We denote the space associated with the operator $\odelta_k$ by
$$
\mathcal{P}^{*,-}_1\Lambda^k:=\star(\mathcal{P}^-_1\Lambda^{n-k})=\mathcal{P}_0\Lambda^k+\star\okappa\star(\mathcal{P}_0\Lambda^{k-1}).
$$

For $\Xi$ a subdomain of $\Omega$, we denote by $E_\Xi^\Omega$ the extension from $L^1_{\rm loc}(\Xi)$ to $L^1_{\rm loc}(\Omega)$, the spaces of locally integrable functions, respectively. Namely, 
$$
E_\Xi^\Omega: L^1_{\rm loc}(\Xi)\to L^1_{\rm loc}(\Omega),\ \ 
E_\Xi^\Omega v=\left\{\begin{array}{ll}
v,&\mbox{on}\,\Xi,
\\
0,&\mbox{else},
\end{array}\right.
\ \mbox{for}\,v\in L^1_{\rm loc}(\Xi).
$$
We use the same notation $L^1_{\rm loc}$ for both scalar and non-scalar locally integrable functions, and, here and in the sequel, use the same notation $E_\Xi^\Omega$ for both scalar and non-scalar functions. 

Let $\mathcal{G}_\Omega=\{\mathcal{G}_h\}$ be a set of shape regular simplicial subdivisions of $\Omega$. On a $\mathcal{G}_h$, define formally the product of a set of function spaces $\left\{\Upsilon(T)\right\}_{T\in\mathcal{G}_h}$ defined cell by cell such that $E_T^\Omega\Upsilon(T)$ for all $T\in\mathcal{G}_h$ are compatible, 
$$
\prod_{T\in\mathcal{G}_h}\Upsilon(T):=\sum_{T\in\mathcal{G}_h} E_T^\Omega\Upsilon(T),
$$
and the summation is direct. The $\displaystyle \prod_{T\in\mathcal{G}_h}\Upsilon(T)$ defined this way is essentially the tensor product of all $\Upsilon(T)$. Denote
$$
\displaystyle\mathcal{P}^-_1\Lambda^k(\mathcal{G}_h):=\prod_{T\in\mathcal{G}_h}\mathcal{P}^-_1\Lambda^k(T); \ \displaystyle\mathcal{P}^{*,-}_1\Lambda^k(\mathcal{G}_h):=\prod_{T\in\mathcal{G}_h}\mathcal{P}^{*,-}_1\Lambda^k(T); \ \displaystyle\mathcal{P}_0\Lambda^k(\mathcal{G}_h)=\prod_{T\in\mathcal{G}_h}\mathcal{P}_0\Lambda^k(T).
$$


Denote the conforming finite element spaces with Whitney forms by
$$
\fW_h\Lambda^k:=\mathcal{P}^-_1(\mathcal{G}_h)\cap H\Lambda^k,\ \ \fW_{h0}\Lambda^k:=\mathcal{P}^-_1(\mathcal{G}_h)\cap H_0\Lambda^k,
$$
and 
$$
\fW^*_h\Lambda^k:=\mathcal{P}^{*,-}_1(\mathcal{G}_h)\cap H^*\Lambda^k,\ \ \fW^*_{h0}\Lambda^k:=\mathcal{P}^{*,-}_1(\mathcal{G}_h)\cap H^*_0\Lambda^k.
$$
Then 
$$
\fW^*_h\Lambda^k=\star \fW_h\Lambda^{n-k},\ \ \mbox{and}\ \ \fW^*_{h0}\Lambda^k=\star \fW_{h0}\Lambda^{n-k}.
$$
Denote 
$$
\hf_h\Lambda^k:=\N(\od^k,\fW_h\Lambda^k)\omp \R(\od^{k-1},\fW_h\Lambda^{k-1}),\ \ \ \hf_{h0}\Lambda^k:=\N(\od^k,\fW_{h0}\Lambda^k)\omp \R(\od^{k-1},\fW_{h0}\Lambda^{k-1}),
$$
$$
\hf^*_h\Lambda^k:=\N(\odelta_k,\fW^*_h\Lambda^k)\omp \R(\odelta_{k+1},\fW^*_h\Lambda^{k+1}),\ \mbox{and}\ \hf^*_{h0}\Lambda^k:=\N(\odelta_k,\fW^*_{h0}\Lambda^k)\omp \R(\odelta_{k+1},\fW^*_{h0}\Lambda^{k+1}).
$$
Then
$$
\hf_h\Lambda^k=\hf^*_{h0}\Lambda^k,\ \ \mbox{and}\ \ \hf_{h0}\Lambda^k=\hf^*_h\Lambda^k.
$$
\begin{lemma}(\cite{Arnold.D2018feec})
$\hf_h\Lambda^k$ and $\hf_{h0}\Lambda^k$ are isomorphic to $\hf\Lambda^k$ and $\hf_0\Lambda^k$, respectively.
\end{lemma}

Following \cite{Zhang.S2022padao-arxiv}, denote the accompanied-by-conforming (ABC) finite element spaces with Whitney forms by
$$
\fW^{\rm abc}_h\Lambda^k:=\left\{\fmu_h\in \mathcal{P}^-_1\Lambda^k(\mathcal{G}_h):\sum_{T\in\mathcal{G}_h}\langle\od^k\fmu_h,\feta_h\rangle_{L^2\Lambda^{k+1}(T)}-\langle\fmu_h,\odelta_{k+1}\feta_h\rangle_{L^2\Lambda^k(T)}=0,\ \forall\,\feta_h\in \fW^*_{h0}\Lambda^{k+1}\right\},
$$
$$
\fW^{\rm abc}_{h0}\Lambda^k:=\left\{\fmu_h\in \mathcal{P}^-_1\Lambda^k(\mathcal{G}_h):\sum_{T\in\mathcal{G}_h}\langle\od^k\fmu_h,\feta_h\rangle_{L^2\Lambda^{k+1}(T)}-\langle\fmu_h,\odelta_{k+1}\feta_h\rangle_{L^2\Lambda^k(T)}=0,\ \forall\,\feta_h\in \fW^*_h\Lambda^{k+1}\right\},
$$
$$
\fW^{*,\rm abc}_h\Lambda^k:=\left\{\fmu_h\in \mathcal{P}^{*,-}_1\Lambda^k(\mathcal{G}_h):\sum_{T\in\mathcal{G}_h}\langle\odelta_k\fmu_h,\ftau_h\rangle_{L^2\Lambda^{k-1}(T)}-\langle\fmu_h,\od^{k-1}\ftau_h\rangle_{L^2\Lambda^k(T)}=0,\ \forall\,\ftau_h\in \fW_{h0}\Lambda^{k-1}\right\},
$$
and
$$
\fW^{*,\rm abc}_{h0}\Lambda^k:=\left\{\fmu_h\in \mathcal{P}^{*,-}_1\Lambda^k(\mathcal{G}_h):\sum_{T\in\mathcal{G}_h}\langle\odelta_k\fmu_h,\ftau_h\rangle_{L^2\Lambda^{k-1}(T)}-\langle\fmu_h,\od^{k-1}\ftau_h\rangle_{L^2\Lambda^k(T)}=0,\ \forall\,\ftau_h\in \fW_h\Lambda^{k-1}\right\}.
$$
Note that, $\fW^{\rm abc}_{h(0)}\Lambda^0$ and $\fW^{*,\rm abc}_{h(0)}\Lambda^n$ are basically the famous lowest-degree Crouzeix-Raviart element spaces \cite{Crouzeix.M;Raviart.P1973}. Besides, we have, for example, 
$$
\fW_h\Lambda^k=\left\{\fmu_h\in\mathcal{P}^-_1\Lambda^k(\mathcal{G}_h):\sum_{T\in\mathcal{G}_h}\langle\od^k\fmu_h,\feta_h\rangle_{L^2\Lambda^{k+1}(T)}-\langle\fmu_h,\odelta_{k+1}\feta_h\rangle_{L^2\Lambda^k(T)}=0,\ \forall\,\feta_h\in \fW^{*,\rm abc}_{h0}\Lambda^{k+1}\right\}.
$$

\begin{lemma}(\cite{Arnold.D;Falk.R;Winther.R2006acta}\cite{Zhang.S2022padao-arxiv})\label{lem:piwfs}
There exists a constant $C_{k,n}$, depending on the regularity of $\mathcal{G}_h$, such that 
$$
\icr(\od^k,\fW_h\Lambda^k)\leqslant C_{k,n},\ \ \mbox{and},\ \ \icr(\od^k,\fW_{h0}\Lambda^k)\leqslant C_{k,n},
$$
and 
$$
\icr(\od^k_h,\fW^{\rm abc}_h\Lambda^k)\leqslant C_{k,n},\ \ \mbox{and},\ \ \icr(\od^k_h,\fW^{\rm abc}_{h0}\Lambda^k)\leqslant C_{k,n}.
$$
\end{lemma}
Here and in the sequel, we use the subscript ``$\cdot_h$" to denote the piecewise operation on $\mathcal{G}_h$.
\begin{remark}
By Lemma \ref{lem:piwfs}, it follows that 
$$
\icr(\odelta_k,\fW^*_h\Lambda^k)\leqslant C_{n-k,n},\ \ \mbox{and},\ \ \icr(\odelta_k,\fW^*_{h0}\Lambda^k)\leqslant C_{n-k,n},
$$
and
$$
\icr(\odelta_{k,h},\fW^{*,\rm abc}_h\Lambda^k)\leqslant C_{n-k,n},\ \ \mbox{and},\ \ \icr(\odelta_{k,h},\fW^{*,\rm abc}_{h0}\Lambda^k)\leqslant C_{n-k,n}.
$$
\end{remark}

Denote
$$
\hf^{*,\rm abc}_h\Lambda^k:=\N(\odelta_{k,h},\fW^{*,\rm abc}_h\Lambda^k)\omp \R(\odelta_{k-1,h},\fW^{*,\rm abc}_h\Lambda^{k-1}),\ \ \hf^{\rm abc}_h\Lambda^k:=\N(\od^k_h,\fW^{\rm abc}_h\Lambda^k)\omp \R(\od^{k-1}_h,\fW^{\rm abc}_h\Lambda^{k-1}),
$$
$$
\hf^{*,\rm abc}_{h0}\Lambda^k:=\N(\odelta_{k,h},\fW^{*,\rm abc}_{h0}\Lambda^k)\omp \R(\odelta_{k-1,h},\fW^{*,\rm abc}_{h0}\Lambda^{k-1}),\ \mbox{and},\ \hf^{\rm abc}_{h0}\Lambda^k:=\N(\od^k_h,\fW^{\rm abc}_{h0}\Lambda^k)\omp \R(\od^{k-1}_h,\fW^{\rm abc}_{h0}\Lambda^{k-1}).
$$

\begin{lemma}[Discrete Poincar\'e-Lefschetz duality, \cite{Zhang.S2022padao-arxiv}]
$$
\hf_h\Lambda^k=\star\hf^{\rm abc}_{h0}\Lambda^{n-k},\ \ \mbox{and}\ \ \hf_{h0}\Lambda^k=\star\hf^{\rm abc}_h\Lambda^{n-k}.
$$
\end{lemma}

\begin{lemma}\cite{Zhang.S2022padao-arxiv}\label{lem:hoddecconst}
The Hodge decompositions hold:
\begin{multline*}
\qquad\mathcal{P}_0\Lambda^k(\mathcal{G}_h)=\R(\od^{k-1},\fW_h\Lambda^{k-1})\opp\hf_h\Lambda^k \opp \R(\odelta_{k+1,h},\fW^{*,\rm abc}_{h0}\Lambda^{k+1})
\\
=\R(\od^{k-1},\fW_{h0}\Lambda^{k-1})\opp\hf_{h0}\Lambda^k \opp \R(\odelta_{k+1,h},\fW^{*,\rm abc}_h\Lambda^{k+1}).\qquad
\end{multline*}
\end{lemma}

\section{Finite element space and Poincar\'e inequality}
\label{sec:fes}

\subsection{A pair of adjoint operators associated with Hodge Laplacian}

Denote 
$$\oT:L^2\Lambda^k\to L^2\Lambda^{k+1}\times L^2\Lambda^{k-1},\ \ \mbox{by}\ \ \oT\fmu:=(\od^k\fmu,\odelta_k\fmu),
$$ 
and 
$$
\aoT:L^2\Lambda^{k+1}\times L^2\Lambda^{k-1}\to L^2\Lambda^k,\ \ \mbox{by}\ \ \aoT(\feta,\ftau):=\odelta_{k+1}\feta+\od^{k-1}\ftau.
$$ 
Then formally, the Hodge Laplacian $\odelta_{k+1}\od^k+\od^{k-1}\odelta_k=\aoT\circ\oT$. 

\begin{lemma}
The operator $(\oT,H\Lambda^k\cap H^*_0\Lambda^k)$ is closed, and its range $\R(\oT,H\Lambda^k\cap H^*_0\Lambda^k)$ is closed.
\end{lemma}

\begin{proof}
We begin with the structure of $H\Lambda^k\cap H\Lambda^*_0\Lambda^k$. Firstly, decompose 
$$
H\Lambda^k=\N(\od^k,H\Lambda^k)\opp (\N(\od^k,H\Lambda^k))^\perp
=\R(\od^{k-1},H\Lambda^{k-1})\opp\hf\Lambda^k\opp (\N(\od^k,H\Lambda^k))^\perp.
$$
By the Helmholtz decomposition, $(\N(\od^k,H\Lambda^k))^\perp\subset H\Lambda^k\cap \N(\odelta_k,H^*_0\Lambda^k)$. Namely, 
$$
\R(\od^k,(\N(\od^k,H\Lambda^k))^\perp)=\R(\od^k,H\Lambda^k),\ \ \mbox{and}\ \R(\odelta_k,(\N(\od^k,H\Lambda^k))^\perp)=\{0\}.
$$
Similarly, decompose 
$H^*_0\Lambda^k=\N(\odelta_k,H^*_0\Lambda^k)\opp (\N(\odelta_k,H^*_0\Lambda^k))^\perp$, and 
$$
\R(\od^k,(\N(\odelta_k,H^*_0\Lambda^k))^\perp)=\{0\},\ \ \mbox{and}\ \R(\odelta_k,(\N(\odelta_k,H^*_0\Lambda^k))^\perp)=\R(\odelta_k,H^*_0\Lambda^k).
$$
It then follows that 
$$
\R(\oT,H\Lambda^k\cap H^*_0\Lambda^k)=\R(\od^k,H\Lambda^k)\times\R(\odelta_k,H^*_0\Lambda^k).
$$
Simultaneously, 
\begin{multline*}
\N(\oT,H\Lambda^k\cap H^*_0\Lambda^k)=\N(\od^k,H\Lambda^k)\cap \N(\odelta_k,H^*_0\Lambda^k)
\\
=[\R(\od^{k-1},H\Lambda^{k-1})\opp \hf\Lambda^k ]\cap [\R(\odelta_{k+1},H^*_0\Lambda^{k+1})\opp\hf^*_0\Lambda^k]=\hf\Lambda^k.
\end{multline*}
The closeness of $(\oT,H\Lambda^k\cap H^*_0\Lambda^k)$ can be proved by definition. The proof is completed. 
\end{proof}

\begin{lemma}
The adjoint operator of $(\oT,H\Lambda^k\cap H^*_0\Lambda^k)$ is $(\aoT,H^*_0\Lambda^{k+1}\times H\Lambda^{k-1})$.
\end{lemma}

\begin{proof}
Firstly, 
$$
\R(\aoT, H^*_0\Lambda^{k+1}\times H\Lambda^{k-1})=\R(\odelta_{k+1},H^*_0\Lambda^{k+1})\opp \R(\od^{k-1},H\Lambda^{k-1}),
$$
and 
$$
\N(\aoT, H^*_0\Lambda^{k+1}\times H\Lambda^{k-1})=\N(\odelta_{k+1},H^*_0\Lambda^{k+1})\times\N(\od^{k-1},H\Lambda^{k-1}).
$$
Namely, 
$$
L^2\Lambda^k=\N(\oT,H\Lambda^k\cap H^*_0\Lambda^k)\opp \R(\aoT, H^*_0\Lambda^{k+1}\times H\Lambda^{k-1})
$$
and 
$$
L^2\Lambda^{k+1}\times L^2\Lambda^{k-1}=\R(\oT,H\Lambda^k\cap H^*_0\Lambda^k)\opp \N(\aoT, H^*_0\Lambda^{k+1}\times H\Lambda^{k-1}). 
$$
Therefore, $(\aoT,H^*_0\Lambda^{k+1}\times H\Lambda^{k-1})$ has the same range and kernel spaces as the adjoint operator of $(\oT,H\Lambda^k\cap H^*_0\Lambda^k)$. Further, it is easy to verify that, for any $\fmu\in H\Lambda^k\cap H^*_0\Lambda^k$, $\feta\in H^*_0\Lambda^{k+1}$, and $\ftau\in H\Lambda^{k-1}$,
\begin{equation}
\langle\od^k\fmu,\feta\rangle_{L^2\Lambda^{k+1}}+\langle\odelta_k\fmu,\ftau\rangle_{L^2\Lambda^{k-1}}-\langle\fmu,(\odelta_{k+1}\feta+\od^{k-1}\ftau)\rangle_{L^2\Lambda^k}=0.
\end{equation}
Hence $(\aoT,H^*_0\Lambda^{k+1}\times H\Lambda^{k-1})$ is the adjoint operator of $(\oT,H\Lambda^k\cap H^*_0\Lambda^k)$. 
\end{proof}

\subsection{Base operator pair for discretization}

We use the following notation: (``$\cdot^{\rm m}$" for minimal)
\begin{itemize}
\item $\pddemL^k(T):=\mathcal{P}_0\Lambda^k(T)+\okappa_T(\mathcal{P}_0\Lambda^{k+1}(T))+\star\okappa_T\star(\mathcal{P}_0\Lambda^{k-1}(T))$;
\item $\pddempdL^k(T):=\pddemL^k(T)+\mathcal{H}^2_{\od}(T)$; \  $\pddempdeL^k(T):=\pddemL^k(T)+\mathcal{H}^2_{\odelta}(T)$;
\item $\mathbf{P}_{\odelta\times\od}^k(T):=\mathcal{P}^{*,-}_1\Lambda^{k+1}(T)\times \mathcal{P}^-_1\Lambda^{k-1}(T).$
\end{itemize}

Direct calculation leads to the lemma below. 
\begin{lemma}\label{lem:cell1-1}
Given $\fmu\in \pddempdeL^k(T)$, $\fmu=0$ if and only if, for any $(\feta,\ftau)\in \mathbf{P}_{\odelta\times\od}^k(T)$, 
$$
\langle\od^k\fomega,\feta\rangle_{L^2\Lambda^{k+1}(T)}+\langle\odelta_k\fomega,\ftau\rangle_{L^2\Lambda^{k-1}(T)}-\langle\fomega,(\odelta_{k+1}\feta+\od^{k-1}\ftau)\rangle_{L^2\Lambda^k(T)}=0.
$$ 
\end{lemma}

\begin{lemma}\label{lem:isocellT}
\begin{enumerate}
\item $\N(\oT,\pddempdeL^k(T))=\R(\aoT,\mathbf{P}_{\odelta\times\od}^k(T))$. 
\item There exists a constant $C_{k,n}>0$, depending on the regularity of $T$, such that 
\begin{multline}\label{eq:baseT}
\inf_{(\feta,\ftau)\in\R(\oT,\pddempdeL^k(T))}\sup_{(\fzeta,\fsigma)\in \N(\aoT,\mathbf{P}_{\odelta\times\od}^k(T))}\frac{\langle\fzeta,\feta\rangle_{L^2\Lambda^{k+1}(T)}+\langle\fsigma,\ftau\rangle_{L^2\Lambda^{k-1}(T)}}{(\|\fzeta\|_{L^2\Lambda^{k+1}(T)}+\|\fsigma\|_{L^2\Lambda^{k-1}(T)})(\|\feta\|_{L^2\Lambda^{k+1}(T)}+\|\ftau\|_{L^2\Lambda^{k-1}(T)})}
\\
=\inf_{(\fzeta,\fsigma)\in \N(\aoT,\mathbf{P}_{\odelta\times\od}^k(T))}\sup_{(\feta,\ftau)\in\R(\oT,\pddempdeL^k(T))}\frac{\langle\fzeta,\feta\rangle_{L^2\Lambda^{k+1}(T)}+\langle\fsigma,\ftau\rangle_{L^2\Lambda^{k-1}(T)}}{(\|\fzeta\|_{L^2\Lambda^{k+1}(T)}+\|\fsigma\|_{L^2\Lambda^{k-1}(T)})(\|\feta\|_{L^2\Lambda^{k+1}(T)}+\|\ftau\|_{L^2\Lambda^{k-1}(T)})}\geqslant C_{k,n}.
\end{multline}
\end{enumerate}
\end{lemma}
\begin{proof}
For the first item, evidently, $\N(\oT,\pddempdeL^k(T))=\mathcal{P}_0\Lambda^k(T)=\R(\aoT,\mathbf{P}_{\odelta\times\od}^k(T))$.

For the second item, note that $\R(\oT,\pddempdeL^k(T))=\mathcal{P}_0\Lambda^{k+1}(T)\times \mathcal{P}^-_1\Lambda^{k-1}(T)$ and \begin{multline*}
\N(\aoT,\mathbf{P}_{\odelta\times\od}^k(T))=\mathcal{P}_0\Lambda^{k+1}(T)\times\mathcal{P}_0\Lambda^{k-1}(T)
\\
\oplus{\rm span}\left\{\left((-1)^{kn}\frac{k}{n-k}\star\okappa_T\star(\dx^{\ixalpha_1}\wedge\dots\wedge\dx^{\ixalpha_k}),\okappa_T(\dx^{\ixalpha_1}\wedge\dots\wedge\dx^{\ixalpha_k})\right),\ \ixalpha\in\mathbb{IX}_{k,n}\right\}.
\end{multline*}
Given $(\feta,\ftau)\in \R(\oT,\pddempdeL^k(T))$, $\feta\in\mathcal{P}_0\Lambda^{k+1}(T)$, and $\ftau=\ftau_0+\ftau_1$, such that $\ftau_0\in\mathcal{P}_0\Lambda^{k-1}(T)$ and $\ftau_1\in\okappa_T(\mathcal{P}_0\Lambda^k(T))$, we choose $\fzeta_0=\feta$, $\fsigma_0=\ftau_0$, and $(\fzeta_1,\fsigma_1)\in {\rm span}\Big\{\Big((-1)^{kn}k/(n-k)\star\okappa_T\star(\dx^{\ixalpha_1}\wedge\dots\wedge\dx^{\ixalpha_k}),\okappa_T(\dx^{\ixalpha_1}\wedge\dots\wedge\dx^{\ixalpha_k})\Big),\ \ixalpha\in\mathbb{IX}_{k,n}\Big\}$ such that $\fsigma_1=\ftau_1$, and set $(\fzeta,\fsigma)=(\fzeta_0+\fzeta_1,\fsigma_0+\fsigma_1)$. Then
$$
\langle\fzeta,\feta\rangle_{L^2\Lambda^{k+1}(T)}+\langle\fsigma,\ftau\rangle_{L^2\Lambda^{k-1}(T)}=\langle\feta,\feta\rangle_{L^2\Lambda^{k+1}(T)}+\langle\ftau,\ftau\rangle_{L^2\Lambda^{k-1}(T)},
$$
and 
$$
\|\fzeta\|_{L^2\Lambda^{k+1}(T)}+\|\fsigma\|_{L^2\Lambda^{k-1}(T)}\leqslant C_{k,n} (\|\feta\|_{L^2\Lambda^{k+1}(T)}+\|\ftau\|_{L^2\Lambda^{k-1}(T)}),
$$
for some $C_{k,n}$ depending on the shape regularity of $T$. This proves 
$$
\inf_{(\feta,\ftau)\in\R(\oT,\pddempdeL^k(T))}\sup_{(\fzeta,\fsigma)\in \N(\aoT,\mathbf{P}_{\odelta\times\od}^k(T))}\frac{\langle\fzeta,\feta\rangle_{L^2\Lambda^{k+1}(T)}+\langle\fsigma,\ftau\rangle_{L^2\Lambda^{k-1}(T)}}{(\|\fzeta\|_{L^2\Lambda^{k+1}(T)}+\|\fsigma\|_{L^2\Lambda^{k-1}(T)})(\|\feta\|_{L^2\Lambda^{k+1}(T)}+\|\ftau\|_{L^2\Lambda^{k-1}(T)})}\geqslant C_{k,n}>0. 
$$
The other part of the assertion follows the same way. The proof is completed.  
\end{proof}
Denote $\displaystyle\pddempdeL^k(\mathcal{G}_h):=\prod_{T\in\mathcal{G}_h}\pddempdeL^k(T)$ and $\displaystyle\mathbf{P}_{\odelta\times\od}^k(\mathcal{G}_h):=\prod_{T\in\mathcal{G}_h}\mathbf{P}_{\odelta\times\od}^k(T).$

\begin{lemma}
The pair $\Big[(\oT_h,\pddempdeL^k(\mathcal{G}_h)):L^2\Lambda^k\to L^2\Lambda^{k+1}\times L^2\Lambda^{k-1}, (\aoT_h,\mathbf{P}_{\odelta\times\od}^k(\mathcal{G}_h)):L^2\Lambda^{k+1}\times L^2\Lambda^{k-1}\to L^2\Lambda^k\Big]$ is a base operator pair. 
\end{lemma}
\begin{proof}
Firstly, given $\fmu_h\in \pddempdeL^k(\mathcal{G}_h)$, $\fmu_h=0$ if and only if 
$$
\sum_{T\in\mathcal{G}_h}
\langle\od^k\fmu_h,\feta_h\rangle_{L^2\Lambda^{k+1}(T)}+\langle\odelta_k\fmu_h,\ftau_h\rangle_{L^2\Lambda^{k-1}(T)}-\langle\fmu_h,(\odelta_{k+1}\feta_h+\od^{k-1}\ftau_h)\rangle_{L^2\Lambda^k(T)}=0,\ \forall\,(\feta_h,\ftau_h)\in \mathbf{P}_{\odelta\times\od}^k(\mathcal{G}_h).
$$ 
Similarly, given $(\feta_h,\ftau_h)\in \mathbf{P}_{\odelta\times\od}^k(\mathcal{G}_h)$, $(\feta_h,\ftau_h)=(0,0)$ if and only if
$$
\sum_{T\in\mathcal{G}_h}
\langle\od^k\fmu_h,\feta\rangle_{L^2\Lambda^{k+1}(T)}+\langle\odelta_k\fmu_h,\ftau\rangle_{L^2\Lambda^{k-1}(T)}-\langle\fmu_h,(\odelta_{k+1}\feta+\od^{k-1}\ftau)\rangle_{L^2\Lambda^k(T)}=0,\ \forall\,\fmu_h\in \pddempdeL^k(\mathcal{G}_h).
$$ 
Therefore, the twisted part of $(\pddempdeL^k(\mathcal{G}_h),\mathbf{P}_{\odelta\times\od}^k(\mathcal{G}_h))$ is the pair itself. 

It is easy to obtain:
$$
\R(\oT_h,\pddempdeL^k(\mathcal{G}_h))=\prod_{T\in\mathcal{G}_h}\R(\oT,\pddempdeL^k(T)),\ \ 
\N(\oT_h,\pddempdeL^k(\mathcal{G}_h))=\prod_{T\in\mathcal{G}_h}\N(\oT,\pddempdeL^k(T)),
$$
$$
\R(\aoT_h,\mathbf{P}_{\odelta\times\od}^k(\mathcal{G}_h))=\prod_{T\in\mathcal{G}_h}\R(\aoT,\mathbf{P}_{\odelta\times\od}^k(T)),\ \mbox{and} \ \  \N(\aoT_h,\mathbf{P}_{\odelta\times\od}^k(\mathcal{G}_h))=\prod_{T\in\mathcal{G}_h}\N(\aoT,\mathbf{P}_{\odelta\times\od}^k(T)).
$$
Therefore, by Lemma \ref{lem:isocellT}, 
$$
\N(\oT_h,\pddempdeL^k(\mathcal{G}_h))=\R(\aoT_h,\mathbf{P}_{\odelta\times\od}^k(\mathcal{G}_h)),
$$
and 
\begin{multline*}
\inf_{(\feta_h,\ftau_h)\in\R(\oT_h,\pddempdeL^k(\mathcal{G}_h))}\sup_{(\fzeta_h,\fsigma_h)\in \N(\aoT_h,\mathbf{P}_{\odelta\times\od}^k(\mathcal{G}_h))}\frac{\langle\fzeta_h,\feta_h\rangle_{L^2\Lambda^{k+1}}+\langle\fsigma_h,\ftau_h\rangle_{L^2\Lambda^{k-1}}}{(\|\fzeta_h\|_{L^2\Lambda^{k+1}}+\|\fsigma_h\|_{L^2\Lambda^{k-1}})(\|\feta_h\|_{L^2\Lambda^{k+1}}+\|\ftau_h\|_{L^2\Lambda^{k-1}})}
\\
=\inf_{(\fzeta_h,\fsigma_h)\in \N(\aoT_h,\mathbf{P}_{\odelta\times\od}^k(\mathcal{G}_h))}\sup_{(\feta_h,\ftau_h)\in\R(\oT_h,\pddempdeL^k(\mathcal{G}_h))}\frac{\langle\fzeta_h,\feta_h\rangle_{L^2\Lambda^{k+1}}+\langle\fsigma_h,\ftau_h\rangle_{L^2\Lambda^{k-1}}}{(\|\fzeta_h\|_{L^2\Lambda^{k+1}}+\|\fsigma_h\|_{L^2\Lambda^{k-1}})(\|\feta_h\|_{L^2\Lambda^{k+1}}+\|\ftau_h\|_{L^2\Lambda^{k-1}})}
\\
=\inf_{T\in\mathcal{G}_h}\left(\inf_{(\fzeta,\fsigma)\in \N(\aoT,\mathbf{P}_{\odelta\times\od}^k(T))}\sup_{(\feta,\ftau)\in\R(\oT,\pddempdeL^k(T))}\frac{\langle\fzeta,\feta\rangle_{L^2\Lambda^{k+1}(T)}+\langle\fsigma,\ftau\rangle_{L^2\Lambda^{k-1}(T)}}{(\|\fzeta\|_{L^2\Lambda^{k+1}(T)}+\|\fsigma\|_{L^2\Lambda^{k-1}(T)})(\|\feta\|_{L^2\Lambda^{k+1}(T)}+\|\ftau\|_{L^2\Lambda^{k-1}(T)})}\right)\geqslant C_{k,n}.
\end{multline*}
Therefore $\Big[(\oT_h,\pddempdeL^k(\mathcal{G}_h)):L^2\Lambda^k\to L^2\Lambda^{k+1}\times L^2\Lambda^{k-1}, (\aoT_h,\mathbf{P}_{\odelta\times\od}^k(\mathcal{G}_h)):L^2\Lambda^{k+1}\times L^2\Lambda^{k-1}\to L^2\Lambda^k\Big]$ is a base operator pair by Definition \ref{def:basepair}. The proof is completed. 
\end{proof}

\begin{lemma}
For $\left[(\oT_h,\sD_h),(\aoT_h,\asD_h)\right]$ partially adjoint based on $[(\oT_h,\pddempdeL^k(\mathcal{G}_h)), (\aoT_h,\mathbf{P}_{\odelta\times\od}^k(\mathcal{G}_h))]$,  if $\icr(\aoT_h,\asD_h)<\infty$,
\begin{equation}
\icr(\oT_h,\sD_h)\leqslant 2\cdot\icr(\oT_h,\pddempdeL^k(\mathcal{G}_h))+\icr(\aoT_h,\asD_h).
\end{equation}
\end{lemma}
\begin{proof}
The lemma can be proved by noting $(\oT,\txM)=(\oT_h,\pddempdeL^k(\mathcal{G}_h))$, $(\aoT,\tyN)=(\aoT_h,\mathbf{P}_{\odelta\times\od}^k(\mathcal{G}_h))$ and $\uxM=\{0\}$ in Theorem \ref{thm:chpi} with $\alpha=1$. 
\end{proof}

\subsection{Finite element space $\fV_{\od\cap\odelta}^{\rm m, +\odelta}\Lambda^k$ and discrete Poincar\'e inequality}
On the subdivision $\mathcal{G}_h$ of $\Omega$, define 
\begin{multline}\label{eq:deffems}
\fV_{\od\cap\odelta}^{\rm m, +\odelta}\Lambda^k:=\Big\{\fmu_h\in\pddempdeL^k(\mathcal{G}_h):
\langle\od^k_h\fmu_h,\feta_h\rangle_{L^2\Lambda^{k+1}}-\langle\fmu_h,\odelta_{k+1,h}\feta_h\rangle_{L^2\Lambda^k}=0,\ \forall\,\feta_h\in\fW^{*,\rm abc}_{h0}\Lambda^{k+1},
\\
\mbox{and}\ \ \langle\odelta_{k,h}\fmu_h,\ftau_h\rangle_{L^2\Lambda^{k-1}}-\langle\fmu_h,\od^{k-1}\ftau_h\rangle_{L^2\Lambda^k}=0,\ \forall\,\ftau_h\in\fW_h\Lambda^{k-1}
\Big\}.
\end{multline}

\begin{lemma}
The pair $\left[(\oT_h,\fV_{\od\cap\odelta}^{\rm m, +\odelta}\Lambda^k),(\aoT_h,\fW^{*,\rm abc}_{h0}\Lambda^{k+1}\times\fW_h\Lambda^{k-1})\right]$ is partially adjoint based on $\left[(\oT_h,\pddempdeL^k(\mathcal{G}_h)), (\aoT_h,\mathbf{P}_{\odelta\times\od}^k(\mathcal{G}_h))\right]$
\end{lemma}

\begin{proof}
By the definition \eqref{eq:deffems}, 
\begin{multline*}
\fV_{\od\cap\odelta}^{\rm m, +\odelta}\Lambda^k=\Big\{\fmu_h\in\pddempdeL^k(\mathcal{G}_h):\langle\oT_h\fmu_h,(\feta_h,\ftau_h)\rangle-\langle\fmu_h,\aoT_h(\feta_h,\ftau_h)\rangle=0,
\\ 
\forall\,(\feta_h,\ftau_h)\in \fW^{*,\rm abc}_{h0}\Lambda^{k+1}\times\fW_h\Lambda^{k-1}\Big\}.\qquad
\end{multline*}
On the other hand, by Lemma \ref{lem:cell1-1}, 
$$
\fW^{*,\rm abc}_{h0}\Lambda^{k+1}\times\fW_h\Lambda^{k-1}=\Big\{(\feta_h,\ftau_h)\in \mathbf{P}_{\odelta\times\od}^k(\mathcal{G}_h):
\langle\oT_h\fmu_h,(\feta_h,\ftau_h)\rangle-\langle\fmu_h,\aoT_h(\feta_h,\ftau_h)\rangle=0,\ \fmu_h\in\fV_{\od\cap\odelta}^{\rm m, +\odelta}\Lambda^k\Big\}.
$$
The assertion follows by Definition \ref{def:pao}. The proof is completed. 
\end{proof}

\begin{lemma}\label{lem:disharm}
$\N(\oT_h,\fV_{\od\cap\odelta}^{\rm m, +\odelta}\Lambda^k)=\hf_h\Lambda^k$.
\end{lemma}
\begin{proof}
By Lemma \ref{lem:hoddecconst}, 
\begin{multline*}
\hf_h\Lambda^k=\Big\{\fmu_h\in\mathcal{P}_0\Lambda^k(\mathcal{G}_h): \langle\fmu_h,\odelta_{k+1,h}\feta_h\rangle_{L^2\Lambda^k}=0,\ \forall\,\feta_h\in\fW^{*,\rm abc}_{h0}\Lambda^{k+1},
\\ 
\mbox{and}\ \langle\fmu_h,\od^{k-1}\ftau_h\rangle_{L^2\Lambda^k}=0,\ \forall\,\ftau_h\in\fW_h\Lambda^{k-1}\Big\}=\left\{\fmu_h\in\mathcal{P}_0\Lambda^k(\mathcal{G}_h): \fmu_h\in \fV_{\od\cap\odelta}^{\rm m, +\odelta}\Lambda^k\right\}. \ \mbox{(by\ \eqref{eq:deffems})}
\end{multline*}
The proof is completed. 
\end{proof}

\begin{lemma}\label{lem:pifes}
There exists a constant $C_{k,n}$, depending on the regularity of $\mathcal{G}_h$, such that 
\begin{equation}
\|\fmu_h\|_{L^2\Lambda^k}\leqslant C_{k,n}(\|\od^k_h\fmu_h\|_{L^2\Lambda^{k+1}}+\|\odelta_{k,h}\fmu_h\|_{L^2\Lambda^{k-1}}),\ \ \forall\,\fmu_h\in \fV_{\od\cap\odelta}^{\rm m, +\odelta}\Lambda^k\omp \hf_h\Lambda^k.
\end{equation}
\end{lemma}
\begin{proof}
By Lemma \ref{lem:hfonT} and by the definition of $\icr$, $\icr(\oT,\pddempdeL^k(T))\leqslant C'_{k,n}h_T$, with $C'_{k,n}$ depending on the regularity of $T$. Noting that $\pddempdeL^k(\mathcal{G}_h)=\prod_{T\in\mathcal{G}_h}\pddempdeL^k(T)$, $\displaystyle\icr(\oT_h,\pddempdeL^k(\mathcal{G}_h))=\sup_{T\in\mathcal{G}_h} \icr(\oT,\pddempdeL^k(T))\leqslant C_{k,n}''>0$, with $C_{k,n}''$ depending on the shape regularity of $\mathcal{G}_h$, and if $\mathcal{G}_h$ is quasi-uniform, $\icr(\oT_h,\pddempdeL^k(\mathcal{G}_h))\leqslant C_{k,n}'''h$. 

By Lemma \ref{lem:piwfs}, noting that $\N(\aoT_h,\fW^{*,\rm abc}_{h0}\Lambda^{k+1}\times\fW_h\Lambda^{k-1})=\N(\odelta_{k+1,h},\fW^{*,\rm abc}_{h0}\Lambda^{k+1})\times \N(\od^{k-1},\fW_h\Lambda^{k-1})$, we have $\icr(\aoT_h,\fW^{*,\rm abc}_{h0}\Lambda^{k+1}\times\fW_h\Lambda^{k-1})\leqslant C_{k,n}''''>0$. Therefore, 
$$
\icr(\oT_h,\fV_{\od\cap\odelta}^{\rm m, +\odelta})\leqslant 2 \icr(\oT_h,\pddempdeL^k(\mathcal{G}_h))+\icr(\aoT_h,\fW^{*,\rm abc}_{h0}\Lambda^{k+1}\times\fW_h\Lambda^{k-1})+ C_{k,n}''\leqslant C_{k,n}.
$$ 
Namely 
$$
\|\fmu_h\|_{L^2\Lambda^k}\leqslant C_{k,n}(\|\od^k_h\fmu_h\|_{L^2\Lambda^{k+1}}+\|\odelta_{k,h}\fmu_h\|_{L^2\Lambda^{k-1}}),\ \ \forall\,\fmu_h\in \fV_{\od\cap\odelta}^{\rm m, +\odelta}\Lambda^k\omp \N(\oT_h,\fV_{\od\cap\odelta}^{\rm m, +\odelta}).
$$
The proof is completed.   
\end{proof}

%
%
%
\section{A primal finite element scheme of the Hodge Laplace problem}

\label{sec:fescheme}

We consider the finite element problem: find $\fomega_h\in \fV_{\od\cap\odelta}^{\rm m, +\odelta}\Lambda^k$, such that 
\begin{equation}\label{eq:modelhldisonefield}
\left\{
\begin{array}{rll}
\langle\fomega_h,\fvarsigma_h\rangle_{L^2\Lambda^k}&=0, &\forall\,\fvarsigma_h\in \hf_h\Lambda^k,
\\  
\langle\od^k_h\fomega_h,\od^k_h\fmu_h\rangle_{L^2\Lambda^{k+1}}+\langle\odelta_{k,h}\fomega_h,\odelta_{k,h}\fmu_h\rangle_{L^2\Lambda^{k-1}}&=\langle\ff-\mathbf{P}_{\hf_h\Lambda^k}\ff,\fmu_h\rangle_{L^2\Lambda^k},& \forall\,\fmu_h\in \fV_{\od\cap\odelta}^{\rm m, +\odelta}\Lambda^k.
\end{array}\right.
\end{equation}
$\mathbf{P}_{\hf_h\Lambda^k}$ denotes the projection to $\hf_h\Lambda^k$. 
\begin{lemma}
The system \eqref{eq:modelhldisonefield} is well posed. 
\end{lemma}
\begin{proof}
The existence of a unique solution to \eqref{eq:modelhldisonefield} is straightforward by Lemma \ref{lem:pifes} the discrete Poincar\'e inequality. Further, it holds that 
$$
\|\fomega_h\|_{L^2\Lambda^k}+\|\od^k_h\fomega_h\|_{L^2\Lambda^{k+1}}+\|\odelta_{k,h}\fomega_h\|_{L^2\Lambda^{k-1}}\leqslant C_{k,n}\frac{\langle\ff,\fmu_h\rangle}{\|\od^k_h\fmu_h\|_{L^2\Lambda^{k+1}}+\|\odelta_{k,h}\fmu_h\|_{L^2\Lambda^{k-1}}},
$$
for any $\fmu_h\in \fV_{\od\cap\odelta}^{\rm m, +\odelta}\omp \hf_h\Lambda^k$. The proof is completed. 
\end{proof}

\subsection{Error estimation}
\label{sec:errorest}

Given $\ff\in L^2\Lambda^k$, an equivalent formulation of \eqref{eq:modelhlori} is to find $\fomega\in H\Lambda^k(\Omega)\cap H^*_0(\Omega)$ and $\fvartheta\in \hf\Lambda^k(\Omega)$, such that 
\begin{equation}\label{eq:modelhl}
\left\{
\begin{array}{c}
\langle\fomega,\fvartheta\rangle_{L^2\Lambda^k}=0,\ \ \forall\,\fvartheta\in \hf\Lambda^k,\ \mbox{and}
\\  
\langle\fvartheta,\fmu\rangle_{L^2\Lambda^k}+\langle\od^k\fomega,\od^k\fmu\rangle_{L^2\Lambda^{k+1}}+\langle\odelta_k\fomega,\odelta_k\fmu\rangle_{L^2\Lambda^{k-1}}=\langle\ff,\fmu\rangle_{L^2\Lambda^k},\ \ \forall\,\fmu\in H\Lambda^k\cap H^*_0\Lambda^k.
\end{array}\right.
\end{equation}

We consider the finite element problem for \eqref{eq:modelhl}: find $\fomega_h\in \fV_{\od\cap\odelta}^{\rm m, +\odelta}\Lambda^k$ and $\fvartheta_h\in \hf_h\Lambda^k $, such that
\begin{equation}\label{eq:modelhldis}
\left\{
\begin{array}{rll}
\langle\fomega_h,\fvarsigma_h\rangle_{L^2\Lambda^k}&=0, &\forall\,\fvarsigma_h\in \hf_h\Lambda^k,
\\  
\langle\fvartheta_h,\fmu_h\rangle_{L^2\Lambda^k}+\langle\od^k_h\fomega_h,\od^k_h\fmu_h\rangle_{L^2\Lambda^{k+1}}+\langle\odelta_{k,h}\fomega_h,\odelta_{k,h}\fmu_h\rangle_{L^2\Lambda^{k-1}}&=\langle\ff,\fmu_h\rangle_{L^2\Lambda^k},& \forall\,\fmu_h\in \fV_{\od\cap\odelta}^{\rm m, +\odelta}\Lambda^k.
\end{array}\right.
\end{equation}
Then evidently \eqref{eq:modelhldisonefield} and \eqref{eq:modelhldis} are equivalent, as the solutions (the part of $\fomega_h$) are equal. The main technical results of this paper are Lemmas \ref{lem:estbyreg} and \ref{lem:ests-reg} below.

\begin{lemma}\label{lem:estbyreg}
Let $(\fvartheta,\fomega)$ and $(\fvartheta_h,\fomega_h)$ be the solutions of \eqref{eq:modelhl} and \eqref{eq:modelhldis}, respectively. Then
\begin{multline}
\|\fomega-\fomega_h\|_{L^2\Lambda^k}+\|\od^k_h(\fomega-\fomega_h)\|_{L^2\Lambda^{k+1}}+\|\odelta_{k,h}(\fomega-\fomega_h)\|_{L^2\Lambda^{k-1}}+\|\fvartheta-\fvartheta_h\|_{L^2\Lambda^k}
\\
\leqslant C(\inf_{\ftau_h\in\fW_h\Lambda^{k-1}}\|\odelta_k\fomega-\ftau_h\|_{\od^{k-1}}+\inf_{\fmu_h\in\fW_h\Lambda^k}\|\fomega-\fmu_h\|_{\od^k}+\inf_{\fvarsigma_h\in\hf_h\Lambda^k}\|\fvartheta-\fvarsigma_h\|_{L^2\Lambda^k}+h\|\ff\|_{L^2\Lambda^k}).
\end{multline}
\end{lemma}

Here and in the sequel, denote $\|\cdot\|_{\od^k}:=(\|\cdot\|_{L^2\Lambda^k}^2+\|\od^k\cdot\|_{L^2\Lambda^{k+1}}^2)^{1/2}$, and similar is $\|\cdot\|_{\odelta_k}$.

A domain $\Omega$ is called $s$-regular if, for some $0<s\leqslant 1$, for $\fomega\in H\Lambda^k(\Omega)\cap H^*_0\Lambda^k(\Omega)$ or $H_0\Lambda^k(\Omega)\cap H^*\Lambda^k(\Omega)$, $\fomega\in H^s\Lambda^k(\Omega)$, and 
\begin{equation}
\|\fomega\|_{H^s\Lambda^k}\leqslant C(\|\fomega\|_{L^2\Lambda^k}+\|\od^k\fomega\|_{L^2\Lambda^{k+1}}+\|\odelta_k\fomega\|_{L^2\Lambda^{k-1}}),
\end{equation}
A smoothly bounded domain is 1-regular and a Lipschitz domain is 1/2-regular. \cite{Arnold.D;Falk.R;Winther.R2006acta} 

\begin{lemma}\cite{Arnold.D;Falk.R;Winther.R2006acta} 
Let $\Omega$ be $s$-regular. Let $\fomega$ be the solution of \eqref{eq:modelhldisonefield}. Then
$$
\|\fomega\|_{H^s\Lambda^k}+\|\od^k\fomega\|_{H^s\Lambda^{k+1}}+\|\odelta_k\fomega\|_{H^s\Lambda^{k-1}}\leqslant C\|\ff\|_{L^2\Lambda^k}.
$$
\end{lemma}

\begin{lemma}\label{lem:ests-reg}
Let $\Omega$ be $s$-regular. Let $(\fvartheta,\fomega)$ and $(\fvartheta_h,\fomega_h)$ be the solutions of \eqref{eq:modelhl} and \eqref{eq:modelhldis}, respectively. Then
\begin{equation}\label{eq:errests-reg}
\|\fomega-\fomega_h\|_{L^2\Lambda^k}+\|\od^k\fomega-\od^k_h\fomega_h\|_{L^2\Lambda^{k+1}}+\|\odelta_k\fomega-\odelta_{k,h}\fomega_h\|_{L^2\Lambda^{k-1}}+\|\fvartheta-\fvartheta_h\|_{L^2\Lambda^k}\leqslant Ch^s\|\ff\|_{L^2\Lambda^k}.
\end{equation}
\end{lemma}

We postpone the proofs of Lemmas \ref{lem:estbyreg} and \ref{lem:ests-reg} into Section \ref{sec:proofs} after some technical preparations. The main result of the paper, the theorem below, follows by Lemmas \ref{lem:estbyreg} and \ref{lem:ests-reg} directly. 

\begin{theorem}
Let $\fomega$ and $\fomega_h$ be the solutions of \eqref{eq:modelhlori} and \eqref{eq:modelhldisonefield}, respectively. Then
\begin{multline}
\|\fomega-\fomega_h\|_{L^2\Lambda^k}+\|\od^k_h(\fomega-\fomega_h)\|_{L^2\Lambda^{k+1}}+\|\odelta_{k,h}(\fomega-\fomega_h)\|_{L^2\Lambda^{k-1}}
\\
\leqslant C(\inf_{\ftau_h\in\fW_h\Lambda^{k-1}}\|\odelta_k\fomega-\ftau_h\|_{\od^{k-1}}+\inf_{\fmu_h\in\fW_h\Lambda^k}\|\fomega-\fmu_h\|_{\od^k}+\inf_{\fvarsigma_h\in\hf_h\Lambda^k}\|\mathbf{P}_{\hf\Lambda^k}\ff-\fvarsigma_h\|_{L^2\Lambda^k}+h\|\ff\|_{L^2\Lambda^k}).
\end{multline}
Let $\Omega$ be $s$-regular. Then
\begin{equation}
\|\fomega-\fomega_h\|_{L^2\Lambda^k}+\|\od^k\fomega-\od^k_h\fomega_h\|_{L^2\Lambda^{k+1}}+\|\odelta_k\fomega-\odelta_{k,h}\fomega_h\|_{L^2\Lambda^{k-1}}\leqslant Ch^s\|\ff\|_{L^2\Lambda^k}.
\end{equation}
\end{theorem}

\begin{remark}
We here discuss the global finite element space based on local shape function space $\pddempdeL^k(T)$ and for the continuous space $H\Lambda^k\cap H^*_0\Lambda^k$. The discussions on global finite element space based on local shape function space $\pddempdL^k(T)$ and/or for the continuous space $H_0\Lambda^k\cap H^*\Lambda^k$ can be carried out in quite a symmetric way. 
\end{remark}

\subsubsection{An auxiliary scheme}
We use the classical mixed method of the Hodge-Laplace problem as an auxiliary scheme. The mixed formulation of the Hodge-Laplace problem reads: find $(\tilde\fvartheta,\tilde\fsigma,\tilde\fomega)\in \hf\Lambda^k\times H\Lambda^{k-1}\times H\Lambda^k$ such that 
\begin{equation}\label{eq:classicalmix}
\left\{
\begin{array}{cccll}
&&\langle\tilde\fomega,\fvarsigma\rangle_{L^2\Lambda^k}&=0&\forall\,\fvarsigma\in\,\hf\Lambda^k
\\
&\langle\tilde\fsigma,\ftau\rangle_{L^2\Lambda^{k-1}}&-\langle\tilde\fomega,\od^{k-1}\ftau\rangle_{L^2\Lambda^k}&=0&\forall\,\ftau\in H\Lambda^{k-1}
\\
\langle\tilde\fvartheta,\fmu\rangle_{L^2\Lambda^k}&+\langle\od^{k-1}\tilde\fsigma,\fmu_h\rangle_{L^2\Lambda^k}&+\langle\od^k\tilde\fomega,\od^k\fmu\rangle_{L^2\Lambda^{k+1}}&=\langle\ff,\fmu\rangle_{L^2\Lambda^k}&\forall\,\fmu\in H\Lambda^k,
\end{array}
\right.
\end{equation}
and the corresponding discretization is to find $(\tilde\fvartheta_h,\tilde\fsigma_h,\tilde\fomega_h)\in \hf_h\Lambda^k\times \fW_h\Lambda^{k-1}\times\fW_h\Lambda^k$ such that 
\begin{equation}\label{eq:classicalmixdis}
\left\{
\begin{array}{cccll}
&&\langle\tilde\fomega_h,\fvarsigma_h\rangle_{L^2\Lambda^k}&=0&\forall\,\fvarsigma_h\in\,\hf_h\Lambda^k
\\
&\langle\tilde\fsigma_h,\ftau_h\rangle_{L^2\Lambda^{k-1}}&-\langle\tilde\fomega_h,\od^{k-1}\ftau_h\rangle_{L^2\Lambda^k}&=0&\forall\,\ftau_h\in\fW_h\Lambda^{k-1}
\\
\langle\tilde\fvartheta,\fmu_h\rangle_{L^2\Lambda^k}&+\langle\od^{k-1}\tilde\fsigma_h,\fmu_h\rangle_{L^2\Lambda^k}&+\langle\od^k\tilde\fomega_h,\od^k\fmu_h\rangle_{L^2\Lambda^{k+1}}&=\langle\ff,\fmu_h\rangle_{L^2\Lambda^k}&\forall\,\fmu_h\in\fW_h\Lambda^k. 
\end{array}
\right.
\end{equation}

\begin{lemma}\label{lem:classicalmix}
Let $(\fvartheta,\fomega)$ and $(\tilde\fvartheta,\tilde\fsigma,\tilde\fomega)$ be the solutions of \eqref{eq:modelhl} and \eqref{eq:classicalmix}, respectively; then
\begin{enumerate}
\item $\fomega=\tilde\fomega$, $\fvartheta=\tilde\fvartheta$, and $\tilde\fsigma=\odelta_k\fomega$;
\item let $(\tilde\fvartheta_h,\tilde\fsigma_h,\tilde\fomega_h)$ be the solution of \eqref{eq:classicalmixdis}; then
\begin{multline*}
\|\tilde\fsigma-\tilde\fsigma_h\|_{\od^{k-1}}+\|\tilde\fomega-\tilde\fomega_h\|_{\od^k}+\|\tilde\fvartheta-\tilde\fvartheta_h\|_{L^2\Lambda^k}
\\
\leqslant C(\inf_{\ftau_h\in\fW_h\Lambda^{k-1}}\|\tilde\fsigma-\ftau_h\|_{\od^{k-1}}+\inf_{\fmu_h\in\fW_h\Lambda^k}\|\tilde\fomega-\fmu_h\|_{\od^k}+\inf_{\fvarsigma_h\in\hf_h\Lambda^k}\|\tilde\fvartheta-\fvarsigma_h\|_{L^2\Lambda^k}+h\|\ff\|_{L^2\Lambda^k});
\end{multline*}
\item for any $\ff\in L^2\Lambda^k$ and $\Omega$ being $s$-regular,
$$
\|\tilde\fsigma-\tilde\fsigma_h\|_{L^2\Lambda^k}+\|\tilde\fomega-\tilde\fomega_h\|_{\od^k}+\|\tilde\fvartheta-\tilde\fvartheta_h\|_{L^2\Lambda^k}\leqslant Ch^s\|\ff\|_{L^2\Lambda^k}.
$$
\end{enumerate}
\end{lemma}
\begin{proof}
The proof can be found in \cite{Arnold.D;Falk.R;Winther.R2006acta}, particularly (7.17), (7.30) and Theorem 7.10 therein. 
\end{proof}
We here note that $\tilde\fvartheta=\mathbf{P}_{\hf\Lambda^k}\ff$. Denote $\mathbf{P}^k_0$ the $L^2$ projection to $\mathcal{P}_0\Lambda^k(\mathcal{G}_h)$. 
\begin{lemma}\label{lem:disequivalent}
Given $\ff_0\in\mathcal{P}_0\Lambda^k(\mathcal{G}_h)$, let $(\tilde\fvartheta_h,\tilde\fsigma_h,\tilde\fomega_h)\in \hf_h\Lambda^k\times \fW_h\Lambda^{k-1}\times\fW_h\Lambda^k$ be such that 
\begin{equation}\label{eq:classdisf0}
\left\{
\begin{array}{cccll}
&&\langle\tilde\fomega_h,\fvarsigma_h\rangle_{L^2\Lambda^k}&=0&\forall\,\fvarsigma_h\in\,\hf_h\Lambda^k
\\
&\langle\tilde\fsigma_h,\ftau_h\rangle_{L^2\Lambda^{k-1}}&-\langle\tilde\fomega_h,\od^{k-1}\ftau_h\rangle_{L^2\Lambda^k}&=0&\forall\,\ftau_h\in\fW_h\Lambda^{k-1}
\\
\langle\tilde\fvartheta,\fmu_h\rangle_{L^2\Lambda^k}&+\langle\od^{k-1}\tilde\fsigma_h,\fmu_h\rangle_{L^2\Lambda^k}&+\langle\od^k\tilde\fomega_h,\od^k\fmu_h\rangle_{L^2\Lambda^{k+1}}&=\langle\ff_0,\fmu_h\rangle_{L^2\Lambda^k}&\forall\,\fmu_h\in\fW_h\Lambda^k,
\end{array}
\right.
\end{equation}
and let $(\bar\fvartheta_h,\bar\fomega_h)\in \hf_h\Lambda^k\times \fV_{\od\cap\odelta}^{\rm m, +\odelta}\Lambda^k$ be such that
\begin{equation}\label{eq:modelhldisff0}
\left\{
\begin{array}{rll}
\langle\bar\fomega_h,\fvarsigma_h\rangle_{L^2\Lambda^k}&=0, &\forall\,\fvarsigma_h\in \hf_h\Lambda^k,
\\  
\langle\bar\fvartheta_h,\fmu_h\rangle_{L^2\Lambda^k}+\langle\od^k_h\bar\fomega_h,\od^k_h\fmu_h\rangle_{L^2\Lambda^{k+1}}+\langle\odelta_{k,h}\bar\fomega_h,\odelta_{k,h}\fmu_h\rangle_{L^2\Lambda^{k-1}}&=\langle\ff_0,\fmu_h\rangle_{L^2\Lambda^k},& \forall\,\fmu_h\in \fV_{\od\cap\odelta}^{\rm m, +\odelta}\Lambda^k.
\end{array}\right.
\end{equation}
Then
\begin{equation}
\bar\fvartheta_h=\tilde\fvartheta_h,\ \ \odelta_{k,h}\bar\fomega_h=\tilde\fsigma_h,\ \ \od^k_h\bar\fomega_h=\od^k\tilde\fomega_h,\ \ \mbox{and}\  \ \mathbf{P}^k_0\bar\fomega_h=\mathbf{P}^k_0\tilde\fomega_h. 
\end{equation}
\end{lemma}

\begin{proof}
Firstly, 
$$
\langle\od^k\tilde\fomega_h,\feta_h\rangle_{L^2\Lambda^{k+1}}-\langle\tilde\fomega_h,\odelta_{k+1,h}\feta_h\rangle_{L^2\Lambda^k}=0,\ \ \forall\,\feta_h\in \fW^{*,\rm abc}_{h0}\Lambda^{k+1},
$$
and it holds with some $\fzeta_h\in \fW^{*,\rm abc}_{h0}\Lambda^{k+1}$ and for any $\fmu_h\in \mathcal{P}^-_1\Lambda^k(\mathcal{G}_h)$ that 
$$
\langle\tilde\fvartheta_h,\fmu_h\rangle_{L^2\Lambda^k}+\langle\od^{k-1}\tilde\fsigma_h,\fmu_h\rangle_{L^2\Lambda^k}+\langle\od^k\tilde\fomega_h,\od^k_h\fmu_h\rangle_{L^2\Lambda^{k+1}}-\langle\fzeta_h,\od^k_h\fmu_h\rangle_{L^2\Lambda^{k+1}}+\langle\odelta_{k+1,h}\fzeta_h,\fmu_h\rangle_{L^2\Lambda^k}=\langle\ff_0,\fmu_h\rangle_{L^2\Lambda^k}.
$$
Now we choose artibrarily $\fmu_h\in\mathcal{P}_0\Lambda^k(\mathcal{G}_h)$, and obtain 
$$
\tilde\fvartheta_h+\od^{k-1}\tilde\fsigma_h+\odelta_{k+1,h}\fzeta_h=\ff_0,
$$
and thus 
$$
\langle\od^k\tilde\fomega_h,\od^k_h\fmu_h\rangle_{L^2\Lambda^{k-1}}-\langle\fzeta_h,\od^k_h\fmu_h\rangle_{L^2\Lambda^{k+1}}=0, \ \ \forall\,\fmu_h\in \mathcal{P}^{*,-}_1(\mathcal{G}_h),
$$
which leads to further that
$$
\od^k\tilde\fomega_h=\mathbf{P}^k_0\fzeta_h. 
$$
Therefore, it holds for $\fvarsigma_h\in\hf_h\Lambda^k$, $\ftau_h\in\fW_h\Lambda^{k-1}$, $\feta_h\in \fW^{*,\rm abc}_{h0}\Lambda^{k+1}$ and $\fmu_h\in\mathcal{P}_0\Lambda^k(\mathcal{G}_h)$ that
\begin{equation}\label{eq:completelymixed}
\left\{
\begin{array}{ccccl}
&&&\langle\mathbf{P}^k_0\tilde\fomega_h,\fvarsigma_h\rangle_{L^2\Lambda^k}&=0
\\
&\langle\tilde\fsigma_h,\ftau_h\rangle_{L^2\Lambda^{k-1}}&&-\langle\mathbf{P}^k_0\tilde\fomega_h,\od^{k-1}\ftau_h\rangle_{L^2\Lambda^k}&=0
\\
&&\langle\mathbf{P}^k_0\fzeta_h,\feta_h\rangle_{L^2\Lambda^{k+1}}&-\langle\mathbf{P}^k_0\tilde\fomega_h,\odelta_{k+1,h}\feta_h\rangle_{L^2\Lambda^k}&=0
\\
\langle\tilde\fvartheta_h,\fmu_h\rangle_{L^2\Lambda^k}&+\langle\od^{k-1}\tilde\fsigma_h,\fmu_h\rangle_{L^2\Lambda^k}&+\langle\odelta_{k+1,h}\fzeta_h,\fmu_h\rangle_{L^2\Lambda^k}&&=\langle\ff_0,\fmu_h\rangle.
\end{array}
\right.
\end{equation}
By Lemma \ref{lem:hoddecconst} the discrete Hodge decomposition,  \eqref{eq:completelymixed} is well-posed. By Lemma \ref{lem:surjecdde}, there exists a unique $\bar\fomega_h\in \pddempdeL^k(\mathcal{G}_h)$, such that 
\begin{equation}
\mathbf{P}^k_0\bar\fomega_h=\mathbf{P}^k_0\tilde\fomega_h,\ \ \odelta_{k,h}\bar\fomega_h=\tilde\fsigma_h,\ \ \mbox{and}\ \ \od^k_h\bar\fomega_h=\mathbf{P}^k_0\fzeta_h=\od^k\tilde\fomega_h.
\end{equation}
Then by \eqref{eq:completelymixed}, $\bar\fomega_h\in \fV_{\od\cap\odelta}^{\rm m, +\odelta}\Lambda^k$, as $\langle\od^k_h\bar\fomega_h,\feta_h\rangle_{L^2\Lambda^{k+1}}-\langle\bar\fomega_h,\odelta_{k+1,h}\feta_h\rangle_{L^2\Lambda^k}=0$, $\forall\,\feta_h\in\fW^{*,\rm abc}_{h0}\Lambda^{k+1}$, and $\langle\odelta_{k,h}\bar\fomega_h,\ftau_h\rangle_{L^2\Lambda^{k-1}}-\langle\bar\fomega_h,\od^{k-1}\ftau_h\rangle_{L^2\Lambda^k}=0$, $\forall\,\ftau_h\in\fW_h\Lambda^{k-1}$. Further, by Lemma \ref{lem:disharm}, it is easy to verify that $(\tilde\fvartheta_h,\bar\fomega_h)$ satisfies \eqref{eq:modelhldisff0}. The proof is completed. 
\end{proof}

\subsubsection{Proofs of Lemmas \ref{lem:estbyreg} and \ref{lem:ests-reg}}
\label{sec:proofs}

\paragraph{\bf Proof of Lemma \ref{lem:estbyreg}}
Denote $\ff_0:=\mathbf{P}^k_0\ff$. Let $(\bar\fvartheta_h,\bar\fomega_h)$ be such that 
\begin{equation}\label{eq:modelhldisbar}
\left\{
\begin{array}{rll}
\langle\bar\fomega_h,\fvarsigma_h\rangle_{L^2\Lambda^k}&=0, &\forall\,\fvarsigma_h\in \hf_h\Lambda^k,
\\  
\langle\bar\fvartheta_h,\fmu_h\rangle_{L^2\Lambda^k}+\langle\od^k_h\bar\fomega_h,\od^k_h\fmu_h\rangle_{L^2\Lambda^{k+1}}+\langle\odelta_{k,h}\bar\fomega_h,\odelta_{k,h}\fmu_h\rangle_{L^2\Lambda^{k-1}}&=\langle\ff_0,\fmu_h\rangle_{L^2\Lambda^k},& \forall\,\fmu_h\in \fV_{\od\cap\odelta}^{\rm m, +\odelta}\Lambda^k.
\end{array}\right.
\end{equation}
Then $\bar\fvartheta_h=\fvartheta_h$, and, for any $\fmu_h\in \fV_{\od\cap\odelta}^{\rm m, +\odelta}\omp \hf_h\Lambda^k$, by Lemma \ref{lem:hfonT},
\begin{multline}\label{eq:nobar-bar}
\|\fomega_h-\bar\fomega_h\|_{L^2\Lambda^k}+\|\od^k_h(\fomega_h-\bar\fomega_h)\|_{L^2\Lambda^{k+1}}+\|\odelta_{k,h}(\fomega_h-\bar\fomega_h)\|_{L^2\Lambda^{k-1}}
\\
\leqslant \frac{\langle\ff-\ff_0,\fmu_h\rangle_{L^2\Lambda^k}}{\|\od^k_h\fmu_h\|_{L^2\Lambda^{k+1}}+\|\odelta_{k,h}\fmu_h\|_{L^2\Lambda^{k-1}}}=\frac{\langle\ff,\fmu_h-\mathbf{P}^k_0\fmu_h\rangle_{L^2\Lambda^k}}{\|\od^k_h\fmu_h\|_{L^2\Lambda^{k+1}}+\|\odelta_{k,h}\fmu_h\|_{L^2\Lambda^{k-1}}}\leqslant Ch\|\ff\|_{L^2\Lambda^k}.
\end{multline}
Let $(\breve\fvartheta_h,\breve\fsigma_h,\breve\fomega_h)\in \hf_h\Lambda^k\times \fW_h\Lambda^{k-1}\times\fW_h\Lambda^k$ be such that 
\begin{equation}\label{eq:modelhldismixbreve}
\left\{
\begin{array}{cccll}
&&\langle\breve\fomega_h,\fvarsigma_h\rangle_{L^2\Lambda^k}&=0&\forall\,\fvarsigma_h\in\,\hf_h\Lambda^k
\\
&\langle\breve\fsigma_h,\ftau_h\rangle_{L^2\Lambda^{k-1}}&-\langle\breve\fomega_h,\od^{k-1}\ftau_h\rangle_{L^2\Lambda^k}&=0&\forall\,\ftau_h\in\fW_h\Lambda^{k-1}
\\
\langle\breve\fvartheta,\fmu_h\rangle_{L^2\Lambda^k}&+\langle\od^{k-1}\breve\fsigma_h,\fmu_h\rangle_{L^2\Lambda^k}&+\langle\od^k\breve\fomega_h,\od^k\fmu_h\rangle_{L^2\Lambda^{k+1}}&=\langle\ff_0,\fmu_h\rangle_{L^2\Lambda^k}&\forall\,\fmu_h\in\fW_h\Lambda^k.
\end{array}
\right.
\end{equation}
Then, by Lemma \ref{lem:disequivalent}, 
\begin{equation}\label{eq:disconnection}
\bar\fvartheta_h=\breve\fvartheta_h,\ \ \odelta_{k,h}\bar\fomega_h=\breve\fsigma_h,\ \ \od^k_h\bar\fomega_h=\od^k\breve\fomega_h,\ \ \mbox{and}\ \  \mathbf{P}^k_0\bar\fomega_h=\mathbf{P}^k_0\breve\fomega_h. 
\end{equation}
Let $(\tilde\fvartheta_h,\tilde\fsigma_h,\tilde\fomega_h)\in \hf_h\Lambda^k\times \fW_h\Lambda^{k-1}\times\fW_h\Lambda^k$ be such that 
\begin{equation}\label{eq:modelhldismixtilde}
\left\{
\begin{array}{cccll}
&&\langle\tilde\fomega_h,\fvarsigma_h\rangle_{L^2\Lambda^k}&=0&\forall\,\fvarsigma_h\in\,\hf_h\Lambda^k
\\
&\langle\tilde\fsigma_h,\ftau_h\rangle_{L^2\Lambda^{k-1}}&-\langle\tilde\fomega_h,\od^{k-1}\ftau_h\rangle_{L^2\Lambda^k}&=0&\forall\,\ftau_h\in\fW_h\Lambda^{k-1}
\\
\langle\tilde\fvartheta,\fmu_h\rangle_{L^2\Lambda^k}&+\langle\od^{k-1}\tilde\fsigma_h,\fmu_h\rangle_{L^2\Lambda^k}&+\langle\od^k\tilde\fomega_h,\od^k\fmu_h\rangle_{L^2\Lambda^{k+1}}&=\langle\ff,\fmu_h\rangle_{L^2\Lambda^k}&\forall\,\fmu_h\in\fW_h\Lambda^k.
\end{array}
\right.
\end{equation}
Then $\breve\fvartheta_h=\tilde\fvartheta_h,$ and, by Lemma \ref{lem:hfonT},
$$
\|\tilde\fomega_h-\breve\fomega_h\|_{\od^k}+\|\tilde\fsigma_h-\breve\fsigma_h\|_{\od^{k-1}}\leqslant C\sup_{\fmu_h\in \fW_h\Lambda^k}\frac{\langle\ff-\ff_0,\fmu_h\rangle_{L^2\Lambda^k}}{\|\fmu_h\|_{\od^k}}=C\sup_{\fmu_h\in \fW_h\Lambda^k}\frac{\langle\ff,\fmu_h-\mathbf{P}^k_0\fmu_h\rangle_{L^2\Lambda^k}}{\|\fmu_h\|_{\od^k}}\leqslant Ch\|\ff\|_{L^2\Lambda^k}.
$$
Now, by Lemma \ref{lem:classicalmix},
\begin{multline}\label{eq:errorfinal}
\|\tilde\fsigma-\tilde\fsigma_h\|_{\od^{k-1}}+\|\tilde\fomega-\tilde\fomega_h\|_{\od^k}+\|\tilde\fvartheta-\tilde\fvartheta_h\|_{L^2\Lambda^k}
\\
\leqslant C(\inf_{\ftau_h\in\fW_h\Lambda^{k-1}}\|\tilde\fsigma-\ftau_h\|_{\od^{k-1}}+\inf_{\fmu_h\in\fW_h\Lambda^k}\|\tilde\fomega-\fmu_h\|_{\od^k}+\inf_{\fvarsigma_h\in\hf_h\Lambda^k}\|\tilde\fvartheta-\fvarsigma_h\|_{L^2\Lambda^k}+h\|\ff\|_{L^2\Lambda^k}).
\end{multline}
and thus
\begin{multline}
\|\tilde\fsigma-\breve\fsigma_h\|_{\od^{k-1}}+\|\tilde\fomega-\breve\fomega_h\|_{\od^k}+\|\tilde\fvartheta-\breve\fvartheta_h\|_{L^2\Lambda^k}
\\
\leqslant C(\inf_{\ftau_h\in\fW_h\Lambda^{k-1}}\|\tilde\fsigma-\ftau_h\|_{\od^{k-1}}+\inf_{\fmu_h\in\fW_h\Lambda^k}\|\tilde\fomega-\fmu_h\|_{\od^k}+\inf_{\fvarsigma_h\in\hf_h\Lambda^k}\|\tilde\fvartheta-\fvarsigma_h\|_{L^2\Lambda^k}+h\|\ff\|_{L^2\Lambda^k}).
\end{multline}
By standard techniques,
$$
\|\fomega-\bar\fomega_h\|_{L^2\Lambda^k}\leqslant \|\fomega-\breve\fomega_h\|_{L^2\Lambda^k}+Ch\|\ff\|_{L^2\Lambda^k}.
$$
Therefore,
\begin{multline*}
\|\fomega-\bar\fomega_h\|_{L^2\Lambda^k}+\|\od^k_h(\fomega-\bar\fomega_h)\|_{L^2\Lambda^{k+1}}+\|\odelta_{k,h}(\fomega-\bar\fomega_h)\|_{L^2\Lambda^{k-1}}
\\
=\|\fomega-\bar\fomega_h\|_{L^2\Lambda^k}+\|\od^k\fomega-\od^k\breve\fomega_h\|_{L^2\Lambda^{k+1}}+\|\tilde\fsigma-\breve\fsigma_h\|_{L^2\Lambda^{k-1}}\ \ (\mbox{by}\ \eqref{eq:disconnection})
\\
\leqslant \|\fomega-\breve\fomega_h\|_{L^2\Lambda^k}+\|\od^k\fomega-\od^k\breve\fomega_h\|_{L^2\Lambda^{k+1}}+\|\tilde\fsigma-\breve\fsigma_h\|_{L^2\Lambda^{k-1}}+Ch\|\ff\|_{L^2\Lambda^{k}}.
\end{multline*}
And finally,
\begin{multline*}
\|\fomega-\fomega_h\|_{L^2\Lambda^k}+\|\od^k_h(\fomega-\fomega_h)\|_{L^2\Lambda^{k+1}}+\|\odelta_{k,h}(\fomega-\fomega_h)\|_{L^2\Lambda^{k-1}}+\|\fvartheta-\fvartheta_h\|_{L^2\Lambda^k}
\\
\leqslant \|\fomega-\breve\fomega_h\|_{L^2\Lambda^k}+\|\od^k\fomega-\od^k\breve\fomega_h\|_{L^2\Lambda^{k+1}}+\|\tilde\fsigma-\breve\fsigma_h\|_{L^2\Lambda^{k-1}}+\|\fvartheta-\breve\fvartheta_h\|_{L^2\Lambda^k}+Ch\|\ff\|_{L^2\Lambda^{k}}\  (\mbox{by}\ \eqref{eq:nobar-bar})
\\
\leqslant C(\inf_{\ftau_h\in\fW_h\Lambda^{k-1}}\|\tilde\fsigma-\ftau_h\|_{\od^{k-1}}+\inf_{\fmu_h\in\fW_h\Lambda^k}\|\tilde\fomega-\fmu_h\|_{\od^k}+\inf_{\fvarsigma_h\in\hf_h\Lambda^k}\|\tilde\fvartheta-\fvarsigma_h\|_{L^2\Lambda^k}+h\|\ff\|_{L^2\Lambda^k}).
\end{multline*}
The proof is completed. 
\qed

\paragraph{\bf Proof of Lemma \ref{lem:ests-reg}}
For general $\ff$, we would have, by Lemma \ref{lem:classicalmix}, with notations defined in the proof of Lemma \ref{lem:estbyreg},
\begin{equation}
\|\tilde\fsigma-\tilde\fsigma_h\|_{L^2\Lambda^{k-1}}+\|\tilde\fomega-\tilde\fomega_h\|_{\od^k}+\|\tilde\fvartheta-\tilde\fvartheta_h\|_{L^2\Lambda^k}\leqslant Ch^s\|\ff\|_{L^2\Lambda^k}.
\end{equation}
Take this into the place of \eqref{eq:errorfinal}, and \eqref{eq:errests-reg} can be obtained by repeating the proof of Lemma \ref{lem:estbyreg}. We omit the details here. 
\qed


%
%
\subsection{Implementation of the scheme: locally supported basis functions of $\fV_{\od\cap\odelta}^{\rm m, +\odelta}\Lambda^k$}

The finite element space $\fV_{\od\cap\odelta}^{\rm m, +\odelta}\Lambda^k$ does not correspond to a ``finite element" defined as Ciarlet's triple \cite{Ciarlet.P1978book}. Though, in this section, we present a set of basis functions of $\fV_{\od\cap\odelta}^{\rm m, +\odelta}\Lambda^k$ which are tightly supported. Therefore, with the space $\hf_h\Lambda^k$ well studied, the finite element scheme can be implemented by the standard routine. 

A general procedure is given in Section \ref{sec:generalproce}, and, for an illustration of the procedure, a two-dimensional example is given in Section \ref{sec:examples}, where we particularly refer to Figures \ref{fig:basisinteriorvertex} and \ref{fig:boundaryvertex} for the illustration of the local supports of the basis functions.

\subsubsection{A general procedure}
\label{sec:generalproce}

On a simplex $T$, denote 
$$
\pddedmpdeL^k(T):=\left\{\fmu\in\pddempdeL^k(T):\langle\odelta_k\fmu,\ftau\rangle_{L^2\Lambda^{k-1}}-\langle\fmu,\od^{k-1}\ftau\rangle_{L^2\Lambda^k}=0,\ \forall\,\ftau\in \mathcal{P}^-_1\Lambda^{k-1}(T)\right\},
$$
and 
$$
\pddedempdeL^k(T):=\left\{\fmu\in\pddempdeL^k(T):\langle\od^k\fmu,\feta\rangle_{L^2\Lambda^{k+1}}-\langle\fmu,\odelta_{k+1}\feta\rangle_{L^2\Lambda^k}=0,\ \forall\,\feta\in \mathcal{P}^{*,-}_1\Lambda^{k+1}(T)\right\}.
$$
Namely, with $\Lambda^{\ixalpha}=\dx^{\ixalpha_1}\wedge\dots\wedge\dx^{\ixalpha_k}$, 
$$
\pddedmpdeL^k(T)=\okappa_T(\mathcal{P}_0\Lambda^{k+1}(T))\oplus\left\{\Lambda^{\ixalpha}+\sum_{\ixalpha'\in \mathbb{IX}_{k,n}}C_{\ixalpha\ixalpha'}\tilde{\fmu}_{\odelta,T}^{\ixalpha'}:\ixalpha\in\mathbb{IX}_{k,n}\right\},
$$
where $C_{\ixalpha\ixalpha'}$ are chosen such that 
$$
\displaystyle\left\langle\odelta_k\left(\Lambda^{\ixalpha}+\sum_{\ixalpha'\in \mathbb{IX}_{k,n}}C_{\ixalpha\ixalpha'}\tilde{\fmu}_{\odelta,T}^{\ixalpha'}\right),\ftau\right\rangle_{L^2\Lambda^{k-1}(T)}-\left\langle\Lambda^{\ixalpha}+\sum_{\ixalpha'\in \mathbb{IX}_{k,n}}C_{\ixalpha\ixalpha'}\tilde{\fmu}_{\odelta,T}^{\ixalpha'},\od^{k-1}\ftau\right\rangle_{L^2\Lambda^k(T)}=0,\ \forall\,\ftau\in \mathcal{P}^-_1\Lambda^{k-1}(T),
$$
and
$$
\pddedempdeL^k(T)=\star\okappa_T\star(\mathcal{P}_0\Lambda^{k-1}(T))+\mathcal{H}^2_{\odelta}(T)
$$
Then $\pddedmpdeL^k(T)$ is unisolvent with respect to $\langle\od^k\fmu,\feta\rangle_{L^2\Lambda^{k+1}}-\langle\fmu,\odelta_{k+1}\feta\rangle_{L^2\Lambda^k}$ for $\feta\in \mathcal{P}^{*,-}_1\Lambda^{k+1}(T)$, $\pddedempdeL^k(T)$ is unisolvent with respect to $\langle\odelta_k\fmu,\ftau\rangle_{L^2\Lambda^{k-1}}-\langle\fmu,\od^{k-1}\ftau\rangle_{L^2\Lambda^k}=0$ for $\ftau\in \mathcal{P}^-_1\Lambda^{k-1}(T)$. Further, 
$$
\pddempdeL^k(T)=\pddedmpdeL^k(T)\oplus \pddedempdeL^k(T).
$$
Denote $\displaystyle\pddedmpdeL^k(\mathcal{G}_h)=\prod_{T\in\mathcal{G}_h}\pddedmpdeL^k(T)$ and $\displaystyle\pddedempdeL^k(\mathcal{G}_h)=\prod_{T\in\mathcal{G}_h}\pddedempdeL^k(T)$. Then
\begin{multline}
\fV_{\od\cap\odelta}^{\rm m, +\odelta}\Lambda^k=\Big\{\fmu_h\in\pddempdeL^k(\mathcal{G}_h):
\langle\od^k_h\fmu_h,\feta_h\rangle_{L^2\Lambda^{k+1}}-\langle\fmu_h,\odelta_{k+1,h}\feta_h\rangle_{L^2\Lambda^k}=0,\ \forall\,\feta_h\in\fW^{*,\rm abc}_{h0}\Lambda^{k+1},
\\
\mbox{and}\ \ \langle\odelta_{k,h}\fmu_h,\ftau_h\rangle_{L^2\Lambda^{k-1}}-\langle\fmu_h,\od^{k-1}\ftau_h\rangle_{L^2\Lambda^k}=0,\ \forall\,\ftau_h\in\fW_h\Lambda^{k-1}
\Big\}
\\
=\left\{\fmu_h\in \pddedmpdeL^k(\mathcal{G}_h): \langle\od^k_h\fmu_h,\feta_h\rangle_{L^2\Lambda^{k+1}}-\langle\fmu_h,\odelta_{k+1,h}\feta_h\rangle_{L^2\Lambda^k}=0,\ \forall\,\feta_h\in\fW^{*,\rm abc}_{h0}\Lambda^{k+1}\right\}
\\
\oplus 
\left\{\fmu_h\in \pddedempdeL^k(\mathcal{G}_h):\langle\odelta_{k,h}\fmu_h,\ftau_h\rangle_{L^2\Lambda^{k-1}}-\langle\fmu_h,\od^{k-1}\ftau_h\rangle_{L^2\Lambda^k}=0,\ \forall\,\ftau_h\in\fW_h\Lambda^{k-1}\right\}:=\fV_{\od}+\fV_{\odelta}.
\end{multline}

Now we figure out the basis functions of $\fV_{\od}$ and $\fV_{\odelta}$ respectively. Their combination is the set of basis functions of $\fV_{\od\cap\odelta}^{\rm m, +\odelta}\Lambda^k$. 

\paragraph{\bf Basis functions of $\fV_{\od}$} Note that 
\begin{multline*}
\fV_{\od}={\Bigg\{}\fmu_h\in \prod_{T\in\mathcal{G}_h}\left[\mathcal{P}_0\Lambda^k(T)+\okappa_T(\mathcal{P}_0\Lambda^{k+1}(T))+\mathcal{H}^2_{\odelta}(T)\right]:
\\
\langle\od^k_h\fmu_h,\feta_h\rangle_{L^2\Lambda^{k+1}}-\langle\fmu_h,\odelta_{k+1,h}\feta_h\rangle_{L^2\Lambda^k}=0,\ \forall\,\feta_h\in\fW^{*,\rm abc}_{h0}\Lambda^{k+1}
\\
\langle\odelta_k\fmu_h,\ftau\rangle_{L^2\Lambda^{k-1}(T)}-\langle\fmu_h,\od^{k-1}\ftau\rangle_{L^2\Lambda^k(T)}=0,\ \forall\,\ftau\in \mathcal{P}^-_1\Lambda^{k-1}(T),\ \forall\,T\in\mathcal{G}_h
\Bigg\}
\\
=\left\{\fmu_h\in \fW_h\Lambda^k+\prod_{T\in\mathcal{G}_h}\mathcal{H}^2_{\odelta}(T):\langle\odelta_k\fmu_h,\ftau\rangle_{L^2\Lambda^{k-1}(T)}-\langle\fmu_h,\od^{k-1}\ftau\rangle_{L^2\Lambda^k(T)}=0,\ \forall\,\ftau\in \mathcal{P}^-_1\Lambda^{k-1}(T),\ \forall\,T\in\mathcal{G}_h\right\}.
\end{multline*}
On any simplex $T$, given $\fmu\in \mathcal{P}_0\Lambda^k(T)+\okappa_T(\mathcal{P}_0\Lambda^{k+1}(T))$, there is always a unique $\fmu'\in \mathcal{H}^2_{\odelta}(T)$, such that $\langle\odelta_k(\fmu+\fmu'),\ftau\rangle_{L^2\Lambda^{k-1}(T)}-\langle\fmu+\fmu',\od^{k-1}\ftau\rangle_{L^2\Lambda^k(T)}=0$, for $\forall\,\ftau\in \mathcal{P}^-_1\Lambda^{k-1}(T)$. Therefore, there is a bijection between $\fV_{\od}$ and $\fW_h\Lambda^k$. This way, the basis functions of $\fV_{\od}$ are determined by this 2-step procedure:
\begin{enumerate}
\item find $\mathbf{B}_{\bf W}^k$ a set of linearly independent basis functions of $\fW_h\Lambda^k$;
\item for every $\fpsi_{\bf W}\in \mathbf{B}_{\bf W}^k$, choose $\tilde\fmu_\psi\in \prod_{T\in\mathcal{G}_h}\mathcal{H}_{\odelta}^2(T)$, such that
$$
\langle\odelta_k\fpsi_{\bf W}+\tilde\fmu_\psi,\ftau\rangle_{L^2\Lambda^{k-1}(T)}-\langle\fpsi_{\bf W}+\tilde\fmu_\psi,\od^{k-1}\ftau\rangle_{L^2\Lambda^k(T)}=0,\ \forall\,\ftau\in \mathcal{P}^-_1\Lambda^{k-1}(T),\ \forall\,T\in\mathcal{G}_h,
$$
and set $\fpsi_{\fV}:=\fpsi_{\bf W}+\tilde\fmu_\psi$.
\end{enumerate}
Then $\{\fpsi_{\fV}\}_{\fpsi_{\bf W}\in\mathbf{B}_{\bf W}^k}$ is a set of basis functions of $\fV_{\od}$. Evidently, the support of $\fpsi_{\fV}$ is contained in the support of $\fpsi_{\bf W}$. 

\paragraph{\bf Basis functions of $\fV_{\odelta}$} To determine the basis functions of $\fV_{\odelta}$, we adopt a different 3-step procedure. 
\begin{description}
\item[Step 1] find $\mathbf{B}_{\bf W}^{k-1}=\{\fpsi_j\}_{j=1}^{\dim(\fW_h\Lambda^{k-1})}$ a set of nodal basis functions of $\fW_h\Lambda^{k-1}$, and on every simplex $T$, the restrictions $\fpsi_j|_T$ of those $\fpsi_j$ that are nonzero on $T$ are linearly independent;

\item[Step 2]
given $T\in\mathcal{G}_h$, set $I^T:=\left\{1\leqslant i\leqslant \dim(\fW_h\Lambda^{k-1}):\mathring{T}\cap {\rm supp}(\fpsi_i)\neq\emptyset\right\}$, and there exist a set of functions $\left\{\fmu^T_i:i\in I^T\right\}\subset \pddedempdeL^k(T)$, such that $\langle\odelta_k\fmu^T_i,\fpsi_j|_T\rangle_{L^2\Lambda^{k-1}(T)}-\langle\fmu^T_i,\od^{k-1}\fpsi_j|_T\rangle_{L^2\Lambda^k(T)}=\delta_{ij}$, $i,j\in I^T$. Then $\pddedempdeL^k(T)={\rm span}\{\fmu^T_i:\ i\in I^T\}$.

\item[Step 3] a set of basis functions of $\fV_{\odelta}$ consists of, for $1\leqslant j\leqslant \dim(\fW_h\Lambda^{k-1})$, functions 
$$
\fmu_h\in\sum_{\mathring{T}\cap{\rm supp}(\fpsi_j)\neq\emptyset}{\rm span}\{E_T^\Omega\fmu_j^T\},\ \ \mbox{such\ that}\ \ \langle\odelta_{k,h}\fmu_h,\fpsi_j\rangle_{L^2\Lambda^{k-1}}-\langle\fmu_h,\od^{k-1}\fpsi_j\rangle_{L^2\Lambda^k}=0.
$$
\end{description}
Actually,
\begin{multline*}
\fV_{\odelta}=\left\{\fmu_h\in \pddedempdeL^k(\mathcal{G}_h):\langle\odelta_{k,h}\fmu_h,\ftau_h\rangle_{L^2\Lambda^{k-1}}-\langle\fmu_h,\od^{k-1}\ftau_h\rangle_{L^2\Lambda^k}=0,\ \forall\,\ftau_h\in\fW_h\Lambda^{k-1}\right\}
\\
=\left\{\fmu_h\in \sum_{T\in\mathcal{G}_h}\sum_{i\in I^T}{\rm span}\{E_T^\Omega\fmu_i^T\}:\langle\odelta_{k,h}\fmu_h,\fpsi_j\rangle_{L^2\Lambda^{k-1}}-\langle\fmu_h,\od^{k-1}\fpsi_j\rangle_{L^2\Lambda^k}=0,\ \forall\,\fpsi_j\in \mathbf{B}_{\bf W}^{k-1}\right\}
\\
=\sum_{1\leqslant j\leqslant \dim(\fW_h\Lambda^{k-1})}\left\{\fmu_h\in\sum_{\mathring{T}\cap{\rm supp}(\fpsi_j)\neq\emptyset}{\rm span}\{E_T^\Omega\fmu_j^T\}:\langle\odelta_{k,h}\fmu_h,\fpsi_j\rangle_{L^2\Lambda^{k-1}}-\langle\fmu_h,\od^{k-1}\fpsi_j\rangle_{L^2\Lambda^k}=0\right\}.
\end{multline*}
Note that, again, the support of such functions are contained in the support of $\fpsi_j$.

\subsubsection{Examples}
\label{sec:examples}
We take the two-dimensional Hodge-Laplacian problem of 1-form for example. Let $\Omega$ be a polygon. Denote by $\mathcal{H}(\Omega)$ the space of harmonic forms. The problem reads: find $\fomega\in H(\rot,\Omega)\cap H_0(\dv,\Omega)$, such that $\fomega\perp \mathcal{H}(\Omega)$, and
\begin{equation}
(\rot\fomega,\rot\fmu)+(\dv\fomega,\dv\fmu)=(\ff-\mathbf{P}_{\mathcal{H}}\ff,\fmu),\ \forall\,\fmu\in H(\rot,\Omega)\cap H_0(\dv,\Omega). 
\end{equation}
The corresponding spaces are $H^1(\Omega)=H({\rm grad},\Omega)$ for 0-forms and $H^1_0(\Omega)=H_0(\curl,\Omega)$ for 2-forms, respectively. We use the conforming linear element space $V^1_h$ for $H({\rm grad},\Omega)$ and the linear Crouzeix-Raviart element space $V^{\rm CR}_{h0}$ for $H_0(\curl,\Omega)=H^1_0(\Omega)$.

Let $\mathcal{T}_h$ be a shape-regular triangular subdivision of $\Omega$ with mesh size $h$, such that $\overline\Omega=\cup_{T\in\mathcal{T}_h}\overline T$, and every boundary vertex is connected to at least one interior vertex. Denote by $\mathcal{E}_h$, $\mathcal{E}_h^i$, $\mathcal{E}_h^b$, $\mathcal{X}_h$, $\mathcal{X}_h^i$, $\mathcal{X}_h^b$ and $\mathcal{X}_h^c$ the set of edges, interior edges, boundary edges, vertices, interior vertices, boundary vertices and corners, respectively. Evidently, $V^1_h$ admits locally supported basis functions, denoted by $\phi_a$ associated with vertices $a\in\mathcal{X}_h$. The restrictions of $\phi_a$ on a triangle are each one of the barycentric coordinates on the triangle. We refer to Figure \ref{fig:vertices} for an illustration of the triangulation, and also the supports of $\phi_a$.

In the setting, 
$$
\pddempdeL^1(T)=\mathbf{P}_{\rot\cap\dv}^{\rm m+\dv}(T)={\rm span}\left\{\left(\begin{array}{c}1\\ 0\end{array}\right),\ \left(\begin{array}{c}0\\ 1\end{array}\right),\ \left(\begin{array}{c}\tilde{x}\\ \tilde{y}\end{array}\right),\ \left(\begin{array}{c}\tilde{y}\\ -\tilde{x}\end{array}\right),\ \left(\begin{array}{c}\tilde{x}^2\\ 0\end{array}\right),\ \left(\begin{array}{c}0\\ \tilde{y}^2\end{array}\right)\right\},
$$
and
$$
\mathbf{P}_{\od\cap\odelta,\odelta}^{\rm m+\odelta}\Lambda^1(T)=\mathbf{P}_{\rot\cap\dv,\dv}^{\rm m+\dv}(T)={\rm span}\left\{\left(\begin{array}{c}\tilde{x}\\ \tilde{y}\end{array}\right),\ \left(\begin{array}{c}\tilde{x}^2\\ 0\end{array}\right),\ \left(\begin{array}{c}0\\ \tilde{y}^2\end{array}\right)\right\}.
$$

\begin{figure}[htbp]
\begin{tikzpicture}[scale=1.1]

\path 	coordinate (a0b0) at (0,0)
coordinate (a0b1) at (0,1)
coordinate (a0b2) at (0,2)
coordinate (a0b3) at (0,3)
coordinate (a0b4) at (0,4)
coordinate (a1b0) at (1,0)
coordinate (a1b1) at (1,1)
coordinate (a1b2) at (1,2)
coordinate (a1b3) at (1,3)
coordinate (a1b4) at (1,4)
coordinate (a2b0) at (2,0)
coordinate (a2b1) at (2,1)
coordinate (a2b2) at (2,2)
coordinate (a2b3) at (2,3)
coordinate (a2b4) at (2,4)
coordinate (a3b0) at (3,0)
coordinate (a3b1) at (3,1)
coordinate (a3b2) at (3,2)
coordinate (a3b3) at (3,3)
coordinate (a3b4) at (3,4)
coordinate (a4b0) at (4,0)
coordinate (a4b1) at (4,1)
coordinate (a4b2) at (4,2)
coordinate (a4b3) at (4,3)
coordinate (a4b4) at (4,4);

\draw[line width=.4pt]  (a0b0) -- (a0b4) ;
\draw[line width=.4pt]  (a1b0) -- (a1b4) ;
\draw[line width=.4pt]  (a2b0) -- (a2b4) ;
\draw[line width=.4pt]  (a3b0) -- (a3b4) ;
\draw[line width=.4pt]  (a4b0) -- (a4b4) ;
\draw[line width=.4pt]  (a0b0) -- (a4b0) ;
\draw[line width=.4pt]  (a0b1) -- (a4b1) ;
\draw[line width=.4pt]  (a0b2) -- (a4b2) ;
\draw[line width=.4pt]  (a0b3) -- (a4b3) ;
\draw[line width=.4pt]  (a0b4) -- (a4b4) ;
\draw[line width=.4pt]  (a0b0) -- (a4b0) ;

\draw[line width=.4pt]  (a0b3) -- (a1b4) ;
\draw[line width=.4pt]  (a0b2) -- (a2b4) ;
\draw[line width=.4pt]  (a0b1) -- (a3b4) ;
\draw[line width=.4pt]  (a1b0) -- (a4b3) ;
\draw[line width=.4pt]  (a2b0) -- (a4b2) ;
\draw[line width=.4pt]  (a3b0) -- (a4b1) ;
\draw[line width=.4pt]  (a0b0) -- (a4b4) ;

\draw[fill] (a3b2) circle [radius=0.05];
\node[below right] at (a3b2) {$A$};
\draw[fill] (a2b1) circle [radius=0.05];
\node[below right] at (a2b1) {$B$};
\draw[fill] (a1b0) circle [radius=0.05];
\node[below] at (a1b0) {$C$};
 
\begin{scope}
\fill[pattern=horizontal lines] (a2b1)--(a3b1)--(a4b2)--(a4b3)--(a3b3)--(a2b2)--(a2b1);
\fill[pattern=vertical lines] (a1b0)--(a2b0)--(a3b1)--(a3b2)--(a2b2)--(a1b1)--(a1b0);
\fill[pattern=north west lines] (a0b0)--(a2b0)--(a2b1)--(a1b1)--(a0b0);
\end{scope}


\path 	coordinate (ca0b0) at (5,0)
coordinate (ca1b0) at (6,0)
coordinate (ca2b0) at (7,0)
coordinate (ca2b1) at (7,1)
coordinate (ca1b1) at (6,1);

\draw[line width=.4pt]  (ca2b1) -- (ca2b0) -- (ca1b0) -- (ca0b0) -- (ca1b1) -- (ca2b1) -- (ca1b0) -- (ca1b1);


\path 	coordinate (ba1b0) at (8,0)
coordinate (ba2b0) at (9,0)
coordinate (ba3b1) at (10,1)
coordinate (ba3b2) at (10,2)
coordinate (ba2b2) at (9,2)
coordinate (ba1b1) at (8,1)
coordinate (ba2b1) at (9,1);

\draw[line width=.4pt]  (ba1b1) -- (ba2b2);
\draw[line width=.4pt]  (ba1b0) -- (ba3b2);
\draw[line width=.4pt]  (ba2b0) -- (ba3b1);
\draw[line width=.4pt]  (ba2b2) -- (ba3b2);
\draw[line width=.4pt]  (ba1b1) -- (ba3b1);
\draw[line width=.4pt]  (ba1b0) -- (ba2b0);
\draw[line width=.4pt]  (ba1b0) -- (ba1b1);
\draw[line width=.4pt]  (ba2b0) -- (ba2b2);
\draw[line width=.4pt]  (ba3b1) -- (ba3b2);


\path 	coordinate (aa2b1) at (11,1)
coordinate (aa3b1) at (12,1)
coordinate (aa4b2) at (13,2)
coordinate (aa4b3) at (13,3)
coordinate (aa3b3) at (12,3)
coordinate (aa2b2) at (11,2)
coordinate (aa3b2) at (12,2);

\draw[line width=.4pt]  (aa2b2) -- (aa3b3);
\draw[line width=.4pt]  (aa2b1) -- (aa4b3);
\draw[line width=.4pt]  (aa3b1) -- (aa4b2);
\draw[line width=.4pt]  (aa3b3) -- (aa4b3);
\draw[line width=.4pt]  (aa2b2) -- (aa4b2);
\draw[line width=.4pt]  (aa2b1) -- (aa3b1);
\draw[line width=.4pt]  (aa2b1) -- (aa2b2);
\draw[line width=.4pt]  (aa3b1) -- (aa3b3);
\draw[line width=.4pt]  (aa4b2) -- (aa4b3);

\begin{scope}
\fill[pattern=horizontal lines] (aa2b1)--(aa3b1)--(aa4b2)--(aa4b3)--(aa3b3)--(aa2b2)--(aa2b1);
\fill[pattern=vertical lines] (ba1b0)--(ba2b0)--(ba3b1)--(ba3b2)--(ba2b2)--(ba1b1)--(ba1b0);
\fill[pattern=north west lines] (ca0b0)--(ca2b0)--(ca2b1)--(ca1b1)--(ca0b0);
\end{scope}

\draw[fill] (aa3b2) circle [radius=0.05];
\node[below right] at (aa3b2) {$A$};
\draw[fill] (ba2b1) circle [radius=0.05];
\node[below right] at (ba2b1) {$B$};
\draw[fill] (ca1b0) circle [radius=0.05];
\node[below] at (ca1b0) {$C$};

\end{tikzpicture}
\caption{Illustration of supports of $\phi_A$, $\phi_B$ and $\phi_C$. $A$ (as well as $B$) denotes an interior vertex, and $C$ denotes a boundary vertex.}\label{fig:vertices}
\end{figure}
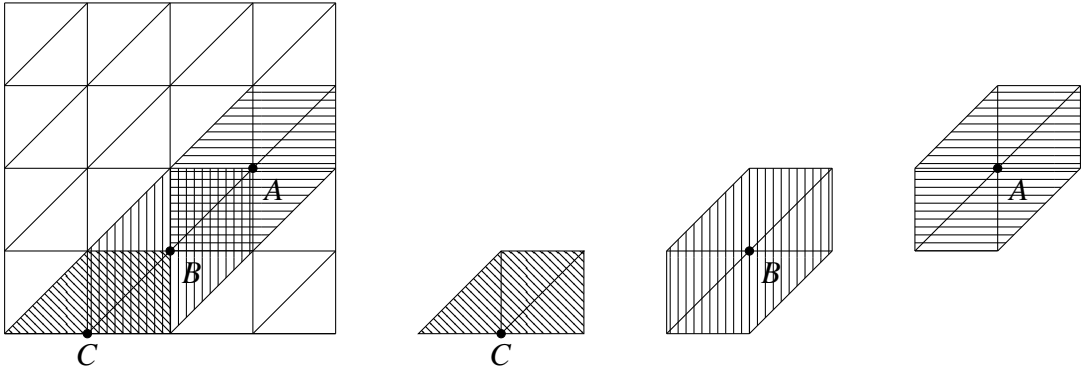

The finite element space is defined by
\begin{multline}\label{eq:defspaceintwod}
\fV_{\od\cap\odelta}^{\rm m, +\odelta}=\fV_{\rot\cap \dv}^{\rm m,+\dv}:=\Big\{\fmu_h\in\mathbf{P}_{\rot\cap\dv}^{\rm m+\dv}(\mathcal{T}_h): (\rot_h\fmu_h,\feta_h)-(\fmu_h,\curl_h\feta_h)=0,\forall\,\feta_h\in V^{\rm CR}_{h0},
\\
\mbox{and}\ \ 
(\dv_h\fmu_h,\ftau_h)-(\fmu_h,{\rm grad}\ftau_h)=0,\forall\,\ftau_h\in V^1_h. 
\Big\}
\end{multline}

Now, following the general procedure, we present the basis functions of $\fV_{\od}=\fV_{\rot}$ and $\fV_{\odelta}=\fV_{\dv}$, respectively. The a set of basis functions of $\fV_{\rot\cap \dv}^{\rm m,+\dv}$ is a direct summation of the set of basis functions of $\fV_{\rot}$ and of $\fV_{\dv}$. 

\paragraph{\bf Basis functions of $\fV_{\rot}$} The basis functions are determined by 2 steps. 
\begin{description}
\item[Step 1] Choose $\{\utau{}_i\}_{i=1}^{\dim(V^{\rm rot}_h)}$ to be a set of nodal basis functions of $V^{\rm rot}_h$, the lowest degree conforming Raviart-Thomas finite element space for $H(\rot,\Omega)$.
\item[Step 2] For any $\utau{}_i$, on every $T\in \mathcal{T}_h$, find $\tilde\fmu_i^T\in {\rm span}\left\{\left(\begin{array}{c}\tilde{x}^2\\ 0\end{array}\right),\ \left(\begin{array}{c}0\\ \tilde{y}^2\end{array}\right)\right\}$, such that 
$$
(\dv(\utau{}_i|_T+\tilde\fmu_i^T),\lambda_j)_T+((\utau{}_i|_T+\tilde\fmu_i^T),\nabla\lambda_j)_T=0,\ \ j=1,2,3. 
$$
Set $\tilde{\utau}{}_i:=\utau{}_i+\prod_{T\in\mathcal{T}_h}\tilde\fmu_i^T$. 
\end{description}
Then $\{\tilde{\utau}{}_i\}_{i=1}^{\dim(V^{\rm rot}_h)}$ is a set of basis functions of $\fV_{\rot}$. The support of $\tilde{\utau}{}_i$ is the same as that of $\utau{}_i$. 

\paragraph{\bf Basis functions of $\fV_{\dv}$} The basis functions are determined by 2 steps. 
\begin{description}
\item[Step 1] On a cell $T$, with $a_i\in\mathcal{X}_h$ being its vertices, let $\lambda^{a_i}_T$ be the barycentric coordinates of $T$, and find $\fmu^{\dv}_{a_i,T}\in \mathbf{P}_{\rot\cap\dv,\dv}^{\rm m+\dv}(T)$, $1\leqslant i\leqslant 3$, such that 
$$
(\dv\fmu^{\dv}_{a_i,T},\lambda_T^{a_j})_T+(\fmu^{\dv}_{a_i,T},{\rm grad}\lambda_T^{a_j})_T=\delta_{ij}, \ \ 1\leqslant i,j\leqslant 3.
$$
\item[Step 2] Find functions in $\prod_{T\in\mathcal{T}_h}{\rm span}\{\fmu^{\dv}_{a,T}:a\in\mathcal{X}_h\}$ such that conditions in \eqref{eq:defspaceintwod} are satisfied with respect to every $\phi_a\in V^1_h$ a basis function. Particularly, with respect to any vertex $a$, the associated basis functions of $\fV_{\dv}$ are all these functions in ${\rm span}\{\fmu^{\dv}_{a,T}:\mathring{T}\cap {\rm supp}(\phi_a)\neq\emptyset\}$, namely, functions of the form $\fomega_h=\sum_{\partial T\ni a}c_TE_T^\Omega\fmu^{\dv}_{a,T}$, such that 
\begin{equation}
\sum_{\partial T\ni a}(\dv\fomega_h,\phi_a|_T)_T+(\fomega_h,\nabla\phi_a|_T)_T=0.
\end{equation}
\end{description}
Those $\fomega_h$ for all $\phi_a$ form a set of basis functions of $\fV_{\dv}$. They are each supported in a two-successive-cell patch. We refer to Figure \ref{fig:basisinteriorvertex} for the case $a\in\mathcal{X}_h^i$, and to Figure \ref{fig:boundaryvertex} for an illustration that $a\in\mathcal{X}_h^b$.

\begin{figure}
\begin{tikzpicture}[scale=1.3]
\path 	
coordinate (a1mb0) at (0.9,0)
coordinate (a1pb0m) at (1.1,-0.1)
coordinate (a1pb0p) at (1.1,0.1)
coordinate (a1mb0m) at (0.9,-0.1)
coordinate (a1mb0p) at (0.9,0.1)
coordinate (a2pb0p) at (2.1,0.1)
coordinate (a2pb0m) at (2.1,-0.1)
coordinate (a2mb0p) at (1.9,0.1)
coordinate (a2mb0m) at (1.9,-0.1)
coordinate (a1b1m) at (1,0.9)
coordinate (a1mb1) at (0.9,1)
coordinate (a1pb1p) at (1.1,1.1)
coordinate (a1pb1m) at (1.1,0.9)
coordinate (a1mb1p) at (0.9,1.1)
coordinate (a1mb1m) at (0.9,0.9)
coordinate (a2b1m) at (2,0.9)
coordinate (a2pb1p) at (2.1,1.1)
coordinate (a2pb1m) at (2.1,0.9)
coordinate (a2mb1p) at (1.9,1.1)
coordinate (a2mb1m) at (1.9,0.9)
coordinate (a3pb1p) at (3.1,1.1)
coordinate (a3pb1m) at (3.1,0.9)
coordinate (a3mb1p) at (2.9,1.1)
coordinate (a3mb1m) at (2.9,0.9)
coordinate (a1mb2) at (0.9,2)
coordinate (a1pb2p) at (1.1,2.1)
coordinate (a1pb2m) at (1.1,1.9)
coordinate (a1mb2p) at (0.9,2.1)
coordinate (a1mb2m) at (0.9,1.9)
coordinate (a2pb2) at (2.1,2)
coordinate (a2b2p) at (2,2.1)
coordinate (a2b2m) at (2,1.9)
coordinate (a2mb2) at (1.9,2)
coordinate (a2pb2p) at (2.1,2.1)
coordinate (a2pb2m) at (2.1,1.9)
coordinate (a2mb2p) at (1.9,2.1)
coordinate (a2mb2m) at (1.9,1.9)
coordinate (a3pb2) at (3.1,2)
coordinate (a3pb2p) at (3.1,2.1)
coordinate (a3pb2m) at (3.1,1.9)
coordinate (a3mb2p) at (2.9,2.1)
coordinate (a3mb2m) at (2.9,1.9)
coordinate (a2b3p) at (2,3.1)
coordinate (a2pb3p) at (2.1,3.1)
coordinate (a2pb3m) at (2.1,2.9)
coordinate (a2mb3p) at (1.9,3.1)
coordinate (a2mb3m) at (1.9,2.9)
coordinate (a3pb3) at (3.1,3)
coordinate (a3b3p) at (3,3.1)
coordinate (a3pb3p) at (3.1,3.1)
coordinate (a3pb3m) at (3.1,2.9)
coordinate (a3mb3p) at (2.9,3.1)
coordinate (a3mb3m) at (2.9,2.9);

\draw[line width=.4pt]  (a1mb1) -- (a1mb2) -- (a2mb2) -- cycle;
\draw[line width=.4pt]  (a1b1m) -- (a2b2m) -- (a2b1m) -- cycle;
\draw[line width=.4pt]  (a1mb2p) -- (a2mb2p) -- (a2mb3p) -- cycle;
\draw[line width=.4pt]  (a2b2p) -- (a2b3p) -- (a3b3p) -- cycle;
\draw[line width=.4pt]  (a2pb2) -- (a3pb2) -- (a3pb3) -- cycle;
\draw[line width=.4pt]  (a2pb2m) -- (a3pb2m) -- (a2pb1m) -- cycle;

\node at(1.5,2.4){$T_1$};

\node at(1,2.9){$\fmu^{\dv}_{A,T_1}$};

\node at(2.3,2.9) {$T_6$};

\node at(2.4,3.4){$\fmu^{\dv}_{A,T_2}$};

\node at(2.9,2.3) {$T_5$};
\node at(3.5,2.4){$\fmu^{\dv}_{A,T_5}$};

\node at(2.4,1.5){$T_4$};
\node at(2.8,1.1){$\fmu^{\dv}_{A,T_4}$};

\node at(1.7,1.1){$T_3$};
\node at(1.7,0.6){$\fmu^{\dv}_{A,T_3}$};

\node at(1.1,1.7){$T_2$};
\node at(0.5,1.6){$\fmu^{\dv}_{A,T_2}$};

\node at(4.25,2){\large$\Longrightarrow$};

\draw[fill] (2,2) circle [radius=0.05];
\node[below right] at (2,2) {$A$};


\path 	coordinate (aa2b1) at (5,1)
coordinate (aa3b1) at (6,1)
coordinate (aa4b2) at (7,2)
coordinate (aa4b3) at (7,3)
coordinate (aa3b3) at (6,3)
coordinate (aa2b2) at (5,2)
coordinate (aa3b2) at (6,2);

\draw[line width=.4pt]  (aa2b2) -- (aa3b3);
\draw[line width=.4pt]  (aa2b1) -- (aa4b3);
\draw[line width=.4pt]  (aa3b1) -- (aa4b2);
\draw[line width=.4pt]  (aa3b3) -- (aa4b3);
\draw[line width=.4pt]  (aa2b2) -- (aa4b2);
\draw[line width=.4pt]  (aa2b1) -- (aa3b1);
\draw[line width=.4pt]  (aa2b1) -- (aa2b2);
\draw[line width=.4pt]  (aa3b1) -- (aa3b3);
\draw[line width=.4pt]  (aa4b2) -- (aa4b3);

\path 	coordinate (ba2b1) at (9,1)
coordinate (ba3b1) at (10,1)
coordinate (ba4b2) at (11,2)
coordinate (ba4b3) at (11,3)
coordinate (ba3b3) at (10,3)
coordinate (ba2b2) at (9,2)
coordinate (ba3b2) at (10,2);

\draw[line width=.4pt]  (ba2b2) -- (ba3b3);
\draw[line width=.4pt]  (ba2b1) -- (ba4b3);
\draw[line width=.4pt]  (ba3b1) -- (ba4b2);
\draw[line width=.4pt]  (ba3b3) -- (ba4b3);
\draw[line width=.4pt]  (ba2b2) -- (ba4b2);
\draw[line width=.4pt]  (ba2b1) -- (ba3b1);
\draw[line width=.4pt]  (ba2b1) -- (ba2b2);
\draw[line width=.4pt]  (ba3b1) -- (ba3b3);
\draw[line width=.4pt]  (ba4b2) -- (ba4b3);


\path 	coordinate (ca2b1) at (1,-2)
coordinate (ca3b1) at (2,-2)
coordinate (ca4b2) at (3,-1)
coordinate (ca4b3) at (3,0)
coordinate (ca3b3) at (2,0)
coordinate (ca2b2) at (1,-1)
coordinate (ca3b2) at (2,-1);

\draw[line width=.4pt]  (ca2b2) -- (ca3b3);
\draw[line width=.4pt]  (ca2b1) -- (ca4b3);
\draw[line width=.4pt]  (ca3b1) -- (ca4b2);
\draw[line width=.4pt]  (ca3b3) -- (ca4b3);
\draw[line width=.4pt]  (ca2b2) -- (ca4b2);
\draw[line width=.4pt]  (ca2b1) -- (ca3b1);
\draw[line width=.4pt]  (ca2b1) -- (ca2b2);
\draw[line width=.4pt]  (ca3b1) -- (ca3b3);
\draw[line width=.4pt]  (ca4b2) -- (ca4b3);

\path 	coordinate (da2b1) at (5,-2)
coordinate (da3b1) at (6,-2)
coordinate (da4b2) at (7,-1)
coordinate (da4b3) at (7,0)
coordinate (da3b3) at (6,0)
coordinate (da2b2) at (5,-1)
coordinate (da3b2) at (6,-1);

\draw[line width=.4pt]  (da2b2) -- (da3b3);
\draw[line width=.4pt]  (da2b1) -- (da4b3);
\draw[line width=.4pt]  (da3b1) -- (da4b2);
\draw[line width=.4pt]  (da3b3) -- (da4b3);
\draw[line width=.4pt]  (da2b2) -- (da4b2);
\draw[line width=.4pt]  (da2b1) -- (da3b1);
\draw[line width=.4pt]  (da2b1) -- (da2b2);
\draw[line width=.4pt]  (da3b1) -- (da3b3);
\draw[line width=.4pt]  (da4b2) -- (da4b3);

\path 	coordinate (ea2b1) at (9,-2)
coordinate (ea3b1) at (10,-2)
coordinate (ea4b2) at (11,-1)
coordinate (ea4b3) at (11,0)
coordinate (ea3b3) at (10,0)
coordinate (ea2b2) at (9,-1)
coordinate (ea3b2) at (10,-1);

\draw[line width=.4pt]  (ea2b2) -- (ea3b3);
\draw[line width=.4pt]  (ea2b1) -- (ea4b3);
\draw[line width=.4pt]  (ea3b1) -- (ea4b2);
\draw[line width=.4pt]  (ea3b3) -- (ea4b3);
\draw[line width=.4pt]  (ea2b2) -- (ea4b2);
\draw[line width=.4pt]  (ea2b1) -- (ea3b1);
\draw[line width=.4pt]  (ea2b1) -- (ea2b2);
\draw[line width=.4pt]  (ea3b1) -- (ea3b3);
\draw[line width=.4pt]  (ea4b2) -- (ea4b3);

\draw[fill] (aa3b2) circle [radius=0.05];
\node[below right] at (aa3b2) {$A$};
\draw[fill] (ba3b2) circle [radius=0.05];
\node[below right] at (ba3b2) {$A$};
\draw[fill] (ca3b2) circle [radius=0.05];
\node[above left] at (ca3b2) {$A$};
\draw[fill] (da3b2) circle [radius=0.05];
\node[above left] at (da3b2) {$A$};
\draw[fill] (ea3b2) circle [radius=0.05];
\node[below right] at (ea3b2) {$A$};

\begin{scope}
\fill[pattern=horizontal lines] (aa2b1)--(aa3b2)--(aa3b3)--(aa2b2)--(aa2b1);
\fill[pattern=horizontal lines] (ba2b1)--(ba3b1)--(ba3b2)--(ba2b2)--(ba2b1);
\fill[pattern=horizontal lines] (ca2b1)--(ca3b1)--(ca4b2)--(ca3b2)--(ca2b1);
\fill[pattern=horizontal lines] (da3b1)--(da4b2)--(da4b3)--(da3b2)--(da3b1);
\fill[pattern=horizontal lines] (ea3b2)--(ea4b2)--(ea4b3)--(ea3b3)--(ea3b2);
\end{scope}

\end{tikzpicture}

\caption{$A$ is an interior vertex; cf. Figure \ref{fig:vertices}. As $A$ is shared by six triangles, five basis functions are associated with the interior vertex $A$. The shadowed parts are respectively the supports of the basis functions.}\label{fig:basisinteriorvertex}
\end{figure}

\begin{figure}[htbp]
\begin{tikzpicture}[scale=1.5]

\path 	coordinate (ca0mb0) at (0.9,0)
coordinate (ca1mb0) at (1.9,0)
coordinate (ca1mb1) at (1.9,1)

coordinate (ca1b0p) at (2,0.1)
coordinate (ca1b1p) at (2,1.1)
coordinate (ca2b1p) at (3,1.1)

coordinate (ca1pb0) at (2.1,0)
coordinate (ca2pb0) at (3.1,0)
coordinate (ca2pb1) at (3.1,1);

\draw[line width=.4pt]  (ca0mb0) -- (ca1mb0) -- (ca1mb1) --cycle;
\draw[line width=.4pt]  (ca1b0p) -- (ca1b1p) -- (ca2b1p) --cycle;
\draw[line width=.4pt]  (ca1pb0) -- (ca2pb0) -- (ca2pb1) --cycle;

\node at(1.6,0.4){$T_1$};
\node at(1.2,0.8){$\fmu^{\dv}_{C,T_1}$};

\node at(2.3,0.7){$T_2$};
\node at(2.4,1.3){$\fmu^{\dv}_{C,T_2}$};

\node at(2.8,0.4){$T_3$};
\node at(3.45,0.5){$\fmu^{\dv}_{C,T_3}$};

\node at (4.2,0.5) {\large$\Longrightarrow$};

\draw[fill] (2,0) circle [radius=0.05];
\node[below] at (2,0) {$C$};

\path 	coordinate (ca0b0) at (5,0)
coordinate (ca1b0) at (6,0)
coordinate (ca2b0) at (7,0)
coordinate (ca2b1) at (7,1)
coordinate (ca1b1) at (6,1);

\draw[line width=.4pt]  (ca2b1) -- (ca2b0) -- (ca1b0) -- (ca0b0) -- (ca1b1) -- (ca2b1) -- (ca1b0) -- (ca1b1);

\path 	coordinate (cra0b0) at (8,0)
coordinate (cra1b0) at (9,0)
coordinate (cra2b0) at (10,0)
coordinate (cra2b1) at (10,1)
coordinate (cra1b1) at (9,1);

\draw[line width=.4pt]  (cra2b1) -- (cra2b0) -- (cra1b0) -- (cra0b0) -- (cra1b1) -- (cra2b1) -- (cra1b0) -- (cra1b1);

\begin{scope}
\fill[pattern=north west lines] (ca0b0)--(ca1b0)--(ca2b1)--(ca1b1)--(ca0b0);
\fill[pattern=north west lines] (cra1b0)--(cra2b0)--(cra2b1)--(cra1b1)--(cra1b0);
\end{scope}

\draw[fill] (ca1b0) circle [radius=0.05];
\node[below] at (ca1b0) {$C$};
\draw[fill] (cra1b0) circle [radius=0.05];
\node[below] at (cra1b0) {$C$};


\end{tikzpicture}

\caption{ C is boundary vertex with a three-cell patch; cf. Figure \ref{fig:vertices}. Two basis functions are associated with the interior vertex $C$. They are each supported on the shadowed parts.}\label{fig:boundaryvertex}
\end{figure}


\begin{thebibliography}{10}

\bibitem{Arnold.D2018feec}
Douglas Arnold.
\newblock {\em Finite element exterior calculus}.
\newblock SIAM, 2018.

\bibitem{Arnold.D;Brezzi.F1985}
Douglas Arnold and Franco Brezzi.
\newblock Mixed and nonconforming finite element methods: implementation, postprocessing and error estimates.
\newblock {\em RAIRO-Mod{\'e}lisation Math{\'e}matique et Analyse
  Num{\'e}rique}, 19(1):7--32, 1985.

\bibitem{Arnold.D;Falk.R;Winther.R2006acta}
Douglas Arnold, Richard Falk, and Ragnar Winther.
\newblock Finite element exterior calculus, homological techniques, and applications.
\newblock {\em Acta Numerica}, 15:1--155, 2006.

\bibitem{Arnold.D;Falk.R;Winther.R2010bams}
Douglas Arnold, Richard Falk, and Ragnar Winther.
\newblock Finite element exterior calculus: from Hodge theory to numerical stability.
\newblock {\em Bulletin of the American Mathematical Society}, 47(2):281--354, 2010.

\bibitem{Barker.M2022thesis}
Mary Barker.
\newblock {\em A Nonconforming Finite Element Method for the 2D Vector Laplacian}.
\newblock PhD thesis, Washington University in St. Louis, 2022.

\bibitem{Barker.M;Cao.S;Stern.A2022arxiv}
Mary Barker, Shuhao Cao, and Ari Stern.
\newblock A nonconforming primal hybrid finite element method for the two-dimensional vector Laplacian.
\newblock {\em arXiv preprint 2206.10567}, 2022.

\bibitem{Bathe.K;Nitikitpaiboon.C;Wang.X1995cs}
KJ~Bathe, C~Nitikitpaiboon, and X~Wang.
\newblock A mixed displacement-based finite element formulation for acoustic fluid--structure interaction.
\newblock {\em Computers \& Structures}, 56(2-3):225--237, 1995.

\bibitem{Bermudez.A;Rodriguez.R1994cmame}
Alfredo Berm{\'u}dez and Rodolfo Rodr{\'\i}guez.
\newblock Finite element computation of the vibration modes of a fluid--solid system.
\newblock {\em Computer Methods in Applied Mechanics and Engineering},
  119(3-4):355--370, 1994.

\bibitem{Brenner.S;Cui.J;Li.F;Sung.L2008nm}
Susanne~C Brenner, Jintao Cui, Fengyan Li, and L-Y Sung.
\newblock A nonconforming finite element method for a two-dimensional $\curl-\curl$ and ${\rm grad}-\dv$ problem.
\newblock {\em Numerische Mathematik}, 109(4):509--533, 2008.

\bibitem{brenner2009quadratic}
Susanne~C Brenner and Li-Yeng Sung.
\newblock A quadratic nonconforming vector finite element for $H({\rm curl}; \Omega)\cap H ({\rm div}; \Omega)$.
\newblock {\em Applied Mathematics Letters}, 22:892--896, 2009.

\bibitem{Brenner.S;Sung.L;Cui.J2008}
Susanne~C Brenner, Li-yeng Sung, and Jintao Cui.
\newblock An interior penalty method for a two dimensional $\curl-\curl$ and ${\rm grad}-\dv$ problem.
\newblock {\em ANZIAM Journal}, 50:C947--C975, 2008.

\bibitem{Ciarlet.P1978book}
Philippe~G Ciarlet.
\newblock {\em The finite element method for elliptic problems}.
\newblock North-Holland, Amsterdam, 1978.

\bibitem{Crouzeix.M;Raviart.P1973}
Michel Crouzeix and P-A Raviart.
\newblock Conforming and nonconforming finite element methods for solving the stationary Stokes equations I. 
\newblock {\em Revue fran{\c{c}}aise d'automatique informatique recherche op{\'e}rationnelle. Math{\'e}matique}, 7(R3):33--75, 1973.
 

\bibitem{Beiroa.L;Brezzi.F;Marini.D;Alessandro.R2018}
Louren{\c{c}}o~Beir{\~a}o da~Veiga, Franco Brezzi, L~Donatella Marini, and Alessandro Russo.
\newblock Virtual element approximations of the vector potential formulation of magnetostatic problems.
\newblock {\em The SMAI Journal of Computational Mathematics}, 4:399--416, 2018.

\bibitem{Demlow.A;Hirani.A2014}
Alan Demlow and Anil~N. Hirani.
\newblock A posteriori error estimates for finite element exterior calculus: The de Rham complex.
\newblock {\em Foundations of Computational Mathematics}, 14(6):1337--1371, 2014.

\bibitem{Fortin.M;Soulie.M1983}
M~Fortin and M~Soulie.
\newblock A non-conforming piecewise quadratic finite element on triangles.
\newblock {\em International Journal for Numerical Methods in Engineering}, 19(4):505--520, 1983.

\bibitem{Hamdi.M;Ousset.Y;Verchery.G1978ijnme}
Mohamed~Ali Hamdi, Yves Ousset, and Georges Verchery.
\newblock A displacement method for the analysis of vibrations of coupled fluid-structure systems.
\newblock {\em International Journal for Numerical Methods in Engineering}, 13(1):139--150, 1978.

\bibitem{Hiptmair.R2002acta}
Ralf Hiptmair.
\newblock Finite elements in computational electromagnetism.
\newblock {\em Acta Numerica}, 11:237--339, 2002.

\bibitem{Hong.Q;Li.Y;Xu.J2022mc}
Qingguo Hong, Yuwen Li, and Jinchao Xu.
\newblock An extended galerkin analysis in finite element exterior calculus.
\newblock {\em Mathematics of Computation}, 91(335):1077--1106, 2022.

\bibitem{Li.Y2019sinum}
Yuwen Li.
\newblock Some convergence and optimality results of adaptive mixed methods in finite element exterior calculus.
\newblock {\em SIAM Journal on Numerical Analysis}, 57(4):2019--2042, 2019.

\bibitem{Liu.W;Zhang.S2022arxiv}
Wenjia Liu and Shuo Zhang.
\newblock A lowest-degree strictly conservative finite element scheme for incompressible stokes problem on general triangulations.
\newblock {\em arXiv preprint, 2108.10522}, 2021.

\bibitem{Marini.L1985sinum}
Luisa~Donatella Marini.
\newblock An inexpensive method for the evaluation of the solution of the lowest order Raviart--Thomas mixed method.
\newblock {\em SIAM Journal on Numerical Analysis}, 22(3):493--496, 1985.

\bibitem{Mirebeau.J2012aml}
Jean-Marie Mirebeau.
\newblock Nonconforming vector finite elements for $H({\rm curl}; \Omega)\cap H({\rm div}; \Omega)$.
\newblock {\em Applied Mathematics Letters}, 25(3):369--373, 2012.

\bibitem{Monk.P2003mono}
Peter Monk.
\newblock {\em Finite element methods for Maxwell's equations}.
\newblock Oxford University Press, 2003.

\bibitem{Park.C;Sheen.D2003}
Chunjae Park and Dongwoo Sheen.
\newblock $P_1$-nonconforming quadrilateral finite element methods for second-order elliptic problems.
\newblock {\em SIAM Journal on Numerical Analysis}, 41(2):624--640, 2003.

\bibitem{Xi.Y;Ji.X;Zhang.S2020jsc}
Yingxia Xi, Xia Ji, and Shuo Zhang.
\newblock A high accuracy nonconforming finite element scheme for Helmholtz transmission eigenvalue problem.
\newblock {\em Journal of Scientific Computing}, 83, 2020.

\bibitem{Xi.Y;Ji.X;Zhang.S2021cicp}
Yingxia Xi, Xia Ji, and Shuo Zhang.
\newblock A simple low-degree optimal finite element scheme for the elastic transmission eigenvalue problem.
\newblock {\em Communications in Computational Physics}, 30:1061--1082, 2021.

\bibitem{Zeng.H;Zhang.C;Zhang.S2020arxiv}
Huilan Zeng, Chen-Song Zhang, and Shuo Zhang.
\newblock Lowest-degree robust finite element scheme for a fourth-order elliptic singular perturbation problem on rectangular grids.
\newblock {\em arXiv preprint, 2006.15804}, 2020.

\bibitem{Zhang.S2020IMA}
Shuo Zhang.
\newblock Minimal consistent finite element space for the biharmonic equation on quadrilateral grids.
\newblock {\em IMA Journal of Numerical Analysis}, 40(2):1390--1406, 2020.

\bibitem{Zhang.S2021SCM}
Shuo Zhang.
\newblock An optimal piecewise cubic nonconforming finite element scheme for the planar biharmonic equation on general triangulations.
\newblock {\em Science China Mathematics}, 64(11):2579--2602, 2021.

\bibitem{Zhang.S2022padao-arxiv}
Shuo Zhang.
\newblock Partially adjoint discretizations of adjoint operators.
\newblock {\em arXiv preprint 2206.12114}, 2022.

\bibitem{Zhang.S2022primalddelta-arxiv}
Shuo Zhang.
\newblock A primal finite element scheme of the $\mathbf{H}(\od)\cap\mathbf{H}(\odelta)$ elliptic problem.
\newblock {\em arXiv preprint 2207.12003}, 2022.

\end{thebibliography}
\end{document}